\documentclass[11pt]{amsart}

\usepackage{amsmath}
\usepackage{amssymb}
\usepackage{amsfonts}
\usepackage{amsthm}
\usepackage{enumerate}
\usepackage{xcolor}
\usepackage{enumerate}
\usepackage{bm}

\DeclareSymbolFont{bbold}{U}{bbold}{m}{n}
\DeclareSymbolFontAlphabet{\mathbbold}{bbold}

 \newtheorem{thm}{Theorem}[section]
 \newtheorem{cor}[thm]{Corollary}
 \newtheorem{lemma}[thm]{Lemma}
 \newtheorem{prop}[thm]{Proposition}

 \theoremstyle{definition}
 \newtheorem{defn}[thm]{Definition}
  \newtheorem{exas}[thm]{Examples}
 \theoremstyle{remark}
 
 \newtheorem{rems}[thm]{Remarks}

\newtheorem{exa}[thm]{Example}

\numberwithin{equation}{section}

\theoremstyle{definition}

\theoremstyle{remark}

\numberwithin{equation}{section}

\newcommand{\nmm}[1]{{\left\vert\kern-0.25ex\left\vert\kern-0.25ex\left\vert #1 
    \right\vert\kern-0.25ex\right\vert\kern-0.25ex\right\vert}}

\newcommand{\vertiii}[1]{{\left\vert\kern-0.25ex\left\vert\kern-0.25ex\left\vert #1 
    \right\vert\kern-0.25ex\right\vert\kern-0.25ex\right\vert}}

\newcommand{\R}{{\mathbb R}}

\newcommand\C{{\mathbb C}}
\newcommand\N{{\mathbb N}}
\newcommand\A{{\mathcal A}}
\newcommand\K{{\mathcal K}}
\newcommand\Af{{\frak A}}
\newcommand\Mf{{\frak M}}
\newcommand\Nf{{\frak N}}
\renewcommand\l{\lambda}
\renewcommand\O{\Omega}
\newcommand\U{\Upsilon}
\newcommand\lb{{\bm\lambda}} 
\newcommand\nub{{\bm\nu}}

\newcommand\B{\mathcal{B}}

\newcommand\D{\mathcal D}
\renewcommand\P{\mathcal{P}}
\newcommand\LM{\mathcal{LM}}
\newcommand\LT{\mathcal{L}}
\renewcommand\r{\mathbf{r}}
\newcommand\z{\mathbf{z}}
\newcommand\w{\mathbf{w}}
\renewcommand\a{\alpha}
\renewcommand\aa{{\bm\alpha}}  
\renewcommand\b{\beta}
\newcommand\bb{{\bm\beta}} 
\newcommand\nunu{{\bm\nu}}    
\renewcommand\t{{\bm t}}
\newcommand\ep{\varepsilon}

\newcommand\wt{\widetilde}
\newcommand\ov{\overline}
\renewcommand\Re{\operatorname{Re}}
\renewcommand\Im{\operatorname{Im}}
\newcommand\1{{{\bm 1}}} 
\newcommand\dom{\operatorname{dom}} 

\begin{document}

\large
\title{Analytic Besov functional calculus for several commuting operators}

\author{Charles Batty}
\address{St.\ John's College, University of Oxford, Oxford OX1 3JP, England}
\email{charles.batty@sjc.ox.ac.uk}

\author{Alexander Gomilko}
\address{Faculty of Mathematics and Computer Science\\
Nicolas Copernicus University\\
Chopin Street 12/18\\
87-100 Toru\'n, Poland}
\email{alex@gomilko.com}

\author{Dominik Kobos}
\address{
USC School of Philosophy\\
Mudd Hall of Philosophy (MHP)\\
3709 Trousdale Parkway\\
Los Angeles, CA 90089-0451\\
U.S.A.}
\email{dkobos@usc.edu}

\author{Yuri Tomilov}
\address{
Institute of Mathematics\\
Polish Academy of Sciences\\
\'Sniadeckich 8\\
00-656 Warsaw, Poland \\
and \\
Faculty of Mathematics and Computer Science\\
Nicolas Copernicus University\\
Chopin Street 12/18\\
87-100 Toru\'n, Poland
}
\email{ytomilov@impan.pl}


\def\today{\number\day \space\ifcase\month\or
 January\or February\or March\or April\or May\or June\or
 July\or August\or September\or October\or November\or December\fi
 \space \number\year}
\date{\today}

\thanks{}


\begin{abstract}
This paper investigates analytic Besov functions of $n$ variables which act on the generators of $n$ commuting $C_0$-semigroups on a Banach space.   The theory for $n=1$ has already been published, and the present paper uses a different approach to that case as well as extending to the cases when $n\ge2$.   It also clarifies some spectral mapping properties and provides some operator norm estimates.
\end{abstract}

\maketitle
\begin{center}
Dedicated with much gratitude to E. Brian Davies on the occasion of his 80th birthday
\end{center}

\section {Introduction}

Let $A$ be an (unbounded) operator on a complex Banach space $X$, and assume that $-A$ is the generator of a bounded $C_0$-semigroup $(e^{-tA})_{t\ge0}$ on $X$.   The so-called Hille-Phillips functional calculus for $A$ assigns a bounded operator $f(A)$ to each function $f$ in the Hille-Phillips (HP) algebra $\LM$.   If $A$ is sectorial of angle less than $\pi/2$, so the semigroup is bounded and holomorphic on a sector in $\C_+$, then there is a theory of a holomorphic calculus for $A$ which may be applied to bounded holomorphic functions on a sector, but there is no guarantee that this calculus assigns a bounded operator to all bounded holomorphic functions.   
However in the sectorial case there is a bounded calculus for $A$ for all functions in the Banach algebra $\B^1$ of analytic Besov functions on $\C_+$, as shown originally in \cite{Vit}.   We have recently developed a theory of a $\B^1$-calculus, which applies to many semigroup generators, but not to all of them.     There is a $\B^1$-calculus for $A$ if $A$ is sectorial of angle less than $\pi/2$ or if $X$ is a Hilbert space and the semigroup $(e^{-tA})_{t\ge0}$ is bounded.  The theory of the $\B^1$-calculus is set out in detail in \cite{BGT} and \cite{BGT2}, but readers may prefer to read the summary in \cite{BGT4} in order to understand the setting of $\B^1$.

Some applications of the $\B^1$-calculus are given in \cite{BGT} and \cite{BGT2}, and others have been found subsequently.   For example, a new variant of the Katznelson-Tzafriri theorem for the $\B^1$-calculus has been proved in \cite{BS}, provided that $A$ has a $\B^1$-calculus.   Estimates involving functions in $\LM$ can be improved by using the $\B^1$-calculus norm, instead of the HP-norm; the difference is considerable in some cases.

Functional calculus is often used to investigate the various notions of joint spectrum for several commuting operators.   Such work has had limited success, and work on functional calculi for two or more generators of $C_0$-semigroups has been sparse.   The Hille-Phillips calculus for bounded $C_0$-semigroups extends to several commuting semigroups in a straightforward fashion.   The holomorphic calculus for sectorial operators can be applied in some cases.   Some partial results for several commuting holomorphic semigroups have been obtained in \cite{ALM1}, \cite{ALM2}, \cite{LLL} and other papers, and a multi-variable version of the Bochner-Phillips functional calculus was set up in \cite{Mir99}.    Some functional calculi for non-holomorphic functions have been introduced for operators with real spectrum (for example, in \cite{Dav}), and versions for several operators have been studied in \cite{AS},  \cite{ASS} and \cite{Mir}.   The $\B^1$-calculus does not apply to generators of bounded $C_0$-groups, except in very special cases.   Nevertheless Arveson's theory of representations of locally abelian groups can be applied to generators of bounded $C_0$-groups (see \cite[Chapter 8]{Dav1}) and to finitely many commuting $C_0$-groups, in ways which help to develop the theory of the $\B^1$-calculus.   In addition, the book \cite[Chapter 3]{ABG} develops a functional calculus for tuples of commuting generators of $C_0$-groups from a different perspective and uses it for the study of commutators and other applications to mathematical physics.

A two-variable version $\B^2$ of $\B^1$ and a theory of the $\B^2$-calculus for two commuting operators have been covered in the thesis \cite{Kob}, using the same methods as in \cite{BGT} and \cite{BGT2}.   In this paper we extend the results to an arbitrary finite number of variables and commuting operators.   Moreover, we use techniques which are quite different from those in \cite{BGT} and \cite{BGT2}.   We base them on a reproducing formula for analytic Besov functions (see Corollary \ref{CorR}), which was first used in \cite{BGT}, and later used in different ways in \cite{BGT2} and \cite{BGT3}.   This formula provides a route to the essential properties of the $\B^n$-calculus:  mapping $\B^n$ into $L(X)$, showing it is a homomorphism, and that it is unique (see Section \ref{sect5}).    This route is more direct, and clearer, than the arguments used in \cite{BGT} and \cite{BGT2}.

We introduce the Banach algebras $\B^n$, and their elementary decomposition into $2^n$ subalgebras, in Section \ref{sect2}.  In the case $n=1$, the two subalgebras were an ideal $\B^1_0$ and the one-dimensional space of constant functions.   In the general case we need more sophisticated notation in our presentation of the decomposition, but the outcomes in this paper are natural generalisations from the one-variable and two-variable cases.   In Section \ref{sect3}, we obtain a reproducing formula for the elementary components of a given function $f\in\B^n$ (Corollary 3.3).   In Section \ref{sect4}, we define the $\B^n$-calculus initially as a function from $\B^n$ to $L(X,X^{**})$, and then we prove in Section \ref{sect5} that it is a functional calculus.   These proofs proceed in an unusual way, as did the corresponding proofs in \cite{BGT}.   In Sections \ref{sect6}-\ref{sect8}, we establish various further properties of the family of $\B$-calculi, including compatibility across operators, spectral inclusions, and operator norm estimates.

\subsection*{Preliminaries}

\noindent{\it Notation:}
For $n\in\N$, we let $I_n:= \{1,2,\dots,n\}$ and $\P_n$ be the power-set of $I_n$.   We denote elements of $\P_n$ by symbols such as  $\O, \Psi$. For $\O \in \P_n$, let $\O^c := I_n \setminus \O$, and $|\O|$ be the cardinality of $\O$.

We use the notation $\R_+$ for the interval $[0,\infty)$, $\C_+$ for the open right half-plane $\{\l\in\C : \Re\l>0\}$ and $\ov\C_+$ for the closed right half-plane.   Real  numbers may typically be represented by $s,t,\a,\b,\gamma$, and complex numbers by $z,\l,\nu,\zeta$.

We use bold print to denote elements of $\R^n$ and $\C_+^n$; for example, $\mathbf{t}$ denotes $(t_1,\dots,t_n)$ in $\R^n$, $\z$ denotes an $n$-tuple  $(z_1,z_2,\dots,z_n)$ in $\C_+^n$, and similarly $\lb$ denotes an $n$-tuple $(\l_1,\l_2,\dots,\l_n)$.

\noindent{\it Holomorphic functions:}
For a holomorphic function $f$ on $\C_+^n$, we let $D_j f$ be the partial derivative of $f$ with respect to the $j$th coordinate variable $z_j$ where $j \in I_n$.  For $\O \in \P_n$, we let $\D_\O$ be the composition of the derivatives $D_j$ for $j \in \O$ (once each).    If $\O=\emptyset$, then $\D_\emptyset f = f$.   For $\O = I_n$, we may write $\D_n$ instead of $\D_{I_n}$.

Let $f \in H^\infty(\C_+^n)$.  We will make use of the following properties:
\begin{itemize}
\item For $\aa \in (0,\infty)^n$, \; $\displaystyle\sup_{\Re \z = \aa} |f(\z)| = \sup_{\Re\z\ge\aa} |f(\z)|$  \quad (Maximum Principle).
\vskip10pt
\item For $\z\in \C_+^n$, \quad $\displaystyle |\D_\O f(\z)| \le \frac{\|f\|_\infty}{2^{|\O|} \prod_{j\in\O} \Re z_j}$ \qquad\; (Cauchy's inequality).
\end{itemize}
Cauchy's inequality extends to higher-order derivatives, by repeated use of first-order cases.

\noindent Several variants of resolvent functions appear frequently, including the following:
\begin{itemize}
\item  For $z, \l \in \C_+$, we put $r_\l(z) = (z+\l)^{-1}$.
\item  For $\l\in\C_+$, $j\in I_n$ and  $\z\in \C_+^n$,  we put $r_{\l,j}(\z) = (z_j+\l)^{-1}$.
\item  For $\z, \lb \in \C_+^n$, we put  $\r_\lb(\z) = \prod_{j=1}^n (z_j+\l_j)^{-1}$.
\item For $z, \l \in \C_+$, we put  $K(z,\l) = (-2/\pi) (z+\l)^{-2}$.
\item For $\z,\lb \in \C_+^n$, we put  $\K_n(\z,\lb) =  \prod_{j=1}^n K(z_j,\l_j)$.
\end{itemize}

\noindent
{\it Measures:}
We may denote Lebesgue measure on $\R_+^n$ by $d\aa$ or $d\t$, and Lebesgue measure on $\C_+^n$ by $dS_n$.   We let
\[
dV(\lambda)=\alpha \, d\beta \, d\alpha,\qquad \l=\alpha+i\beta\in \C_{+},
\]
and $dV_n$ be the $n$-fold product measure of $dV$, so
\[
dV_n(\lb)= \left( \textstyle{\prod}_{j=1}^n \a_j \right) \, dS_n(\lb), \qquad \lb = (\a_j+i\b_j)_{j=1}^n \in \C_+^n.
\]

\noindent
{\it Operators:}
In this paper an ``operator'' may be either a bounded linear operator on a Banach space $X$ or an unbounded operator {\bf with dense domain} in $X$.   The notation $L(X)$ and $L(X,Y)$ denote the spaces of bounded linear operators from $X$ to $X$ and from $X$ to $Y$, respectively.     Bounded operators will typically be denoted by $T$ and unbounded operators by $A$.

An operator $A$ on $X$ is sectorial of angle less than $\pi/2$ if and only if the spectrum $\sigma(A)$ of $A$ is contained in $\C_+ \cup\{0\}$ and $M_A:= \sup_{\l\in\C_+} \|\l(\l+A)^{-1}\| < \infty$.   Equivalently, the operator $-A$ is the generator of a sectorially bounded holomorphic $C_0$-semigroup on $X$.

\section {The algebra $\B^n$} \label{sect2}

Let $f$ be a holomorphic function on $\C_+^n$,  $\O\in\P_n$ and $k=|\O|$.  Let  $H_\O f$ be the function of variables $\aa_\O := (\a_j)_{j\in \O} \in (0,\infty)^k$  for $j \in \O$,  defined by
\begin{equation} \label{Hdef}
(H_\O f)(\aa_\O) := \sup_{\z \in W_{\aa_\O}} \left| (\D_\O f)(\z) \right|,
\end{equation}
where 
\[
W_{\aa_\O} := \left\{ \z\in\C_+^n:  \text{$\Re z_j = \a_j$ for all $j \in \O$} \right\}.
\]
If $f$ is bounded and $\a_j>0$ for each $j\in\O$, then Cauchy's inequality and the maximum principle imply that the supremum in \eqref{Hdef} is finite, and $H_\O f$ is a decreasing function of each variable $\a_j$ for $j \in \O$.

We define
\begin{equation} \label{BOnorm}
\|f\|_{\B^n_\O} := \int_{\R_+^k} (H_\O f)(\aa_\O) \, d\aa_\O,
\end{equation}
where $d\aa_\O$ denotes Lebesgue measure on $\R_+^k$.   This integral may be infinite.   When $\O$ is the empty set, $H_\emptyset f$ and $\|f\|_{\B^n_\emptyset}$ are both equal to the $H^\infty$-norm $\|f\|_\infty$ of $f$.  

We define $\B^n$ to be the vector space of all holomorphic functions $f$ on $\C_+^n$ such that $\|f\|_{\B^n_\O}$ is finite for all $\O \in \P_n$, so that $H_\O f$ is integrable over $\R_+^k$.  Then $\|\cdot\|_{\B^n_\O}$ is a seminorm on $\B^n$, and there is a norm on $\B^n$ defined by
\begin{equation} \label{Bnnorm}
\|f\|_{\B^n} := \sum_{\O \in \P_n} \|f\|_{\B_\O^n}.
\end{equation}
For $n=1$, this definition agrees with the norm on $\B^1$ defined in \cite[p.33]{BGT}, and, for $n=2$, it agrees with the norm on $\B^2$ defined in \cite[p.37]{Kob}.

  The definitions of $\B^n$ and its norm are clearly invariant under permutation of the variables.   For ease of presentation, in a proof of any statement involving one set $\Omega \in \P_n$, we may choose to assume that $\Omega=I_k$ or $\O^c = I_k$.

A useful property of functions $f\in\B^n$ is that they are bounded and uniformly continuous on $\C_+^n$, so they extend to bounded, uniformly continuous functions on $\ov{\C}_+^n$.  This was noted for $n=1$ in \cite[Proposition 2.1]{BGT} and shown for $n=2$ in \cite[Proposition 3.2.4]{Kob}.  In general, the uniform continuity can be proved by elementary methods (as in \cite{Kob}) or by a special case of an argument in the proof of Remark \ref{3.5}(1).   As a preliminary taste, we will show here that both properties can be deduced using only the finiteness of the norms $\|f\|_{\B^n_\O}$ for every singleton subset $\Omega$ of $I_n$.   However this is an exceptional case, and we need to include the $H^\infty$-norm in the definition of $\|f\|_{\B^n}$ in order that Proposition \ref{3.2.1} is true and later proofs can be simplified.  

\begin{prop} \label{3.2.4}
Let $f$ be a holomorphic function on $\C_+^n$, and assume that $\|f\|_{\B^n_\O}$ is finite for every singleton subset $\Omega$ of $I_n$.    Then $f$ is bounded and uniformly continuous on $\C_+^n$.
\end{prop}

\begin{proof}
Let $\z$ and $\z'$ be points in $\C_+^n$.   To estimate $|f(\z)-f(\z')|$, we will consider the variables one at a time.   The first step will be from $\z$ to $\z^{[1]} = (z_1',z_2,\dots,z_n)$, as follows.

Let $\theta\in(0,\pi/2)$ and $\kappa = \sec\theta>1$.   Consider the following path $\Gamma_1$ in $\C_+$ from $z_1$ to $z_1'$:
\begin{enumerate}[\rm(a)]
\item the line-segment from $z_1$ to $z_1'$ if the gradient is in $[-\cot\theta,\cot\theta]$,
\item otherwise, two line-segments, one from $z_1$ and the other from $z_1'$, one with gradient $\cot\theta$ and the other with gradient $-\theta$, chosen so that the two lines cross at a point $z_1''$ to the right of $z$ and $z_1'$.
\end{enumerate}
Then
\[
|f(\z)-f(\z^{[1]})| = \left| \int_{\Gamma_1} (D_1f)(\l,z_2,\dots,z_n) \,d\l \right| 
\le \kappa \int_{J_1}  H_{\{1\}}(t) \,dt.
\]
Here $J_1$ represents the interval in $\R_+$ between $\Re z_1$ and $\Re z_1'$ in case (a).  In case (b), it represents both the interval between $\Re z_1$ and $\Re z_1''$ and the interval between $\Re z_1'$ and $\Re z_1''$, with the two integrals added together.   
%Hence,
%\[
%|f(\z)- f(\z^{[1]})| \le 2 \kappa \|f\|_{\B^n_{\{1\}}} .
%\]
There is a similar estimate for $|f(\z^{[1]})-f(\z^{[2]})|$, where $\z^{[2]} = (z_1',z_2',z_3, \dots, z_n)$, and so on.   Eventually we obtain
\[
|f(\z) - f(\z')| \le \kappa \sum_{j=1}^n \int_{J_j} H_{\{j\}}(t)\,dt \le 2\kappa \sum_{j=1}^n  \|f\|_{\B^n_{\{j\}}} < \infty.
\]
By fixing $\z'$, this proves that $f$ is bounded.    By letting $\theta\to0$, one may conclude that the estimate also holds for $\kappa=1$.

To establish uniform continuity, we fix $\theta \in (0,\pi/2)$ and thereby we fix $\kappa>1$.    Take $\delta>0$ and assume that $|\z-\z'|<\delta$.    Then the length of each interval involved in the sets $J_j$ is at most $\delta \cot\theta$.    Since the functions $H_{\{j\}}$ are integrable over $\R_+$,  there exists $\delta>0$ so small that all the integrals used above are less than $\ep$, and then we can conclude that $|f(\z)-f(\z')| < 2 \kappa n \ep$ whenever $|\z-\z'|<\delta$.
\end{proof}

We choose the norm $\|\cdot\|_{\B^n}$  because it makes $\B^n$ into a Banach algebra.   There are numerous norms on $\B^n$ which are equivalent to $\|\cdot \|_{\B^n}$.  Examples for $n=1$ may be found in \cite[p.33]{BGT} and for $n=2$ in \cite[Section 3.4]{Kob}.

\begin{prop} \label{3.2.1}
The normed space $(\B^n,\|\cdot\|_{\B^n})$ is a Banach algebra.
\end{prop}

\begin{proof}
The proof of completeness of the space is a straightforward, but uninteresting, extension of the proofs for $\B^1$ in \cite[Proposition 2.3]{BGT} and $\B^2$ in \cite[Proposition 3.2.1]{Kob}.

To show that $\B^n$ is a Banach algebra, we need to show that, for any $f,g \in \B^n$, the product $fg$ satisfies $\|fg\|_{\B^n} \le \|f\|_{\B^n}\|g\|_{\B^n}$.   

Let $\Omega \in \P_n$.  By the product rule,
\[\D_\O(fg) = \sum_{\Psi \subseteq \O} (\D_\Psi  f) (\D_{\O \setminus \Psi}g).
\]
For $\Psi \subseteq \Omega$, 
\begin{align*}
\lefteqn{\null\hskip-30pt \sup \{|(\D_\Psi f)(\D_{\O \setminus \Psi}g)(\z)| : \z\in \C_+^n, \Re z_j = \a_j \,(j \in \Omega)\} \hskip20pt} \\
&\null\hskip15pt\le
\sup \{|(\D_\Psi f)(\z))| :\z \in \C_+^n,  \Re z_j = \a_j \,(j \in \Psi)\} \\
&\null\hskip30pt \times
\sup \{|(\D_{\O \setminus \Psi}g)(\z)| : \z \in \C_+^n, \Re z_j = \a_j \,(j \in \O \setminus \Psi)\}.
\end{align*}
The variables $\a_j$ for $j \in \Psi$ separate from those for $j \in \O\setminus\Psi$ on the right-hand side, so integration with respect to Lebesgue measure on $\R_+^n$ shows that
\[
\|fg\|_{\B^n_\O}   \le \sum_{\Psi \subseteq \O} \|f\|_{\B^n_\Psi} \|g\|_{\B^n_{\O\setminus \Psi}}.
\]
Hence
\begin{align*}
\|fg\|_{\B^n} \le \sum_{\O\in\P_n} \sum_{\Psi \subseteq \O} \|f\|_{\B^n_\Psi} \|g\|_{\B^n_{\O \setminus \Psi}} \le \sum_{\Psi\in\P_n} \|f\|_{\B^n_\Psi} \sum_{\Psi\in\P_n} \|g\|_{\B^n_\Psi} = \|f\|_{\B^n} \|g\|_{\B^n}.  
\end{align*}
\vskip-36pt
\end{proof}
\vskip18pt

We note the following facts, which extend the corresponding statements for $n=1$ or $n=2$.     The first result is almost immediate from the definitions.   The other proofs are either similar to those in \cite{BGT} and \cite{Kob}, or, in some cases, they can be deduced from those results.  

\begin{prop} \label{3.2.0}
  Let $f \in \B^n$, $\O\in\mathcal{P}_n$ with $|\O|=k$, and $\zeta_j \in \C_+$ be fixed for $j \in \O^c$.   For $\z \in \C_+^n$, let $\z_\O = (z_j)_{j\in\O}$, and  $\z_\O^+ = (z^*_j)_{j\in I_n}$,
 where 
 \[
 z^{*}_j = \begin{cases} z_j, & j\in\O, \\  \zeta_j, &j\in\O^c.  \end{cases}
 \]
 Let $\Psi$ be a subset of $\O^c$, and
 \[
 g(\z_\O) = (\D_{\Psi} f)(\z_\O^{+}).
 \]
Then $g \in \B^k$.   In particular, the function $\z_\Omega \mapsto f(\z_\O^+)$ is in $\B^k$.
\end{prop}

\begin{proof}
It suffices to consider the case when $\O=I_{n-1}$ and show that $g \in \B^{n-1}$, in the cases when $\Psi$ is empty or $\Psi=\{n\}$.   Then, we can conclude the result for general $\O$ by the symmetry of the variables, and considering the variables in $\O^c$ successively.

Let $\Upsilon \in\P_{n-1}$.   If $\Psi$ is empty, then it is clear that $(H_{\U}g)(\aa_{\U}) \le (H_{\U}f)(\aa_{\U})$, and so $g \in \B^{n-1}$.   In the case when $\Psi=\{n\}$, we consider $\zeta_n\in\C_+$, $\U\in\P_{n-1}$ with $|\U|=m$ and $\U^+ = \U \cup\{n\}$.
Since $f \in \B^n$, the function
\[
(\aa_\U,\a_n) \mapsto   (H_{\U^+}f) (\aa_\U,\a_n)
\]
 is integrable over $\R_+^{m+1}$.  By Fubini's theorem,  the function
 \[
 \aa_\U \mapsto (H_{\U^+}f)(\aa_\U,\a_n)
 \]
 is integrable over $\R_+^{m}$ for almost all $\a_n \in (0,\infty)$.   
 By the maximum principle, these functions of $\aa_\U$  form a family of non-negative functions which are non-increasing with respect to $\a_n$, and the functions are integrable over $\R_+^{m}$ for almost all $\a_n>0$, so they are integrable for all $\a_n>0$.   
 Moreover, 
 \[
 (H_\U g)(\aa_\U) \le (H_{\U^+}f)(\aa_\U, \Re \zeta_n),
\]
so $H_\U g$ is integrable.
\end{proof}

Let $f \in \B^n$, and $\O \in \P_n$.  The following result was shown in \cite[Proposition 2.2]{BGT} for $n=1$, and in \cite[Proposition 3.2.2]{Kob} for $n=2$.    The existence of iterated limits follows by applying the one-variable result repeatedly, but some additional arguments are needed to obtain a limit over several variables simultaneously.

\begin{prop} \label{3.2.2}
Let $f \in \B^n$, $\O \in \P_n$, $|\O|=k$, and $\z_\Omega = (z_j)_{j\in\O} \in \C_+^k$.   Then the following limit exists in $\C$:
\[
f_\O(\z_\O) :=\lim_{\Re z_j \to \infty, j \in \O^c} f(\z).
\]
Moreover, $f_\O \in \B^k$, and the map $f \mapsto f_\Omega$ is bounded and linear from $\B^n$ to $\B^k$.
\end{prop}

\begin{proof}
By symmetry of the variables, it suffices to consider the case when $\O=I_k$.   First we show the result when $k=n-1$.     For this case, we use the notation $\z$ to denote $(z_1,\dots,z_{n-1}) \in \C_+^{n-1}$, and $\zeta$ for the variable $z_n$.

By Proposition \ref{3.2.0}, for each $\z\in\C_+^{n-1}$ the function $\zeta \mapsto f(\z,\zeta)$ is in $\B^{1}$.   By \cite[Proposition 2.2]{BGT}, we know that
\[
g_1(\z) := \lim_{\Re\zeta\to\infty} f(\z,\zeta)
\]
exists.   In fact,
\begin{align}  \label{zet}
\left| f(\z,\zeta) - g_1(\z) \right| 
&\le \left| \int_{\Re\zeta}^\infty (D_n f)(\z,t + i\Im \zeta) \,dt \right| 
\le \int_{\Re \zeta}^\infty (H_{\{n\}}f)(t) \,dt \to 0, 
\end{align}
uniformly with respect to $\z\in\C_+^{n-1}$, as $\Re \zeta \to\infty$.   
 It follows from Vitali's theorem for several variables that $g_1$ is holomorphic, and the partial derivatives of $f(\cdot,\zeta)$ converge (pointwise) to the corresponding derivatives of $g_1$ as $\Re\zeta\to\infty$. 
   This implies that $H_\U g_1 \le H_\U f$ on $(0,\infty)^{|\U|}$ for all $\U\in\P_{n-1}$, and therefore $g_1 \in \B^{n-1}$ and  $\|g_1\|_{\B^{n-1}} \le \|f\|_{\B^n}$.
 
 If $\O=I_k$ where $k \in\{0,1,\dots,n-2\}$, one may apply the argument above to $g_1$ to show that there is a function $g_2 \in \B^{n-2}$ such that
 \[
 \lim_{\Re \zeta\to\infty} \big(g_1(\z',\zeta)  - g_2(\z')\big) = 0,
 \]
 uniformly for $\z' \in \C_+^{n-2}$.   In combination with \eqref{zet}, it follows that the statement is proved for $\Omega = I_{n-2}$.   Continuing in this way, it can be proved for $\O=I_k$. \end{proof}

\begin{rems} \label{3.2.+}
1. 
In Proposition \ref{3.2.2},  $f_{I_n} = f$ and $f_\emptyset$ is the constant function $\lim_{\Re z_j\to\infty,\,j\in I_n} f(\z)$. 

\noindent
2.  Note that the limits involved in Proposition \ref{3.2.2} are as $\Re z_j \to \infty$.    Functions $f \in \B^n$ may not have limits as $|z_j|\to\infty$.

\noindent 
3. The function $f_\Omega$ in Proposition \ref{3.2.2} may be denoted in various ways.   In particular, we may replace the variable $z_j$ by the symbol $\infty$ for each $j \in\O^c$, and then regard $f$ as being defined on an extended domain. 
Proposition \ref{3.2.0} remains true if $\zeta_j=\infty$ for some $j \in \O^c$.   The function $f_\O$ is holomorphic, by a simple application of Vitali's theorem, and then the rest of the proof remains valid.  
\end{rems}
 
We denote by $\B_0^n$ the ideal of $\B^n$ consisting of those functions $f \in \B^n$ such that  $f_\O$ is identically zero for all proper subsets $\O$ of $I_n$.   It suffices that
\[
\lim_{\Re z_j \to \infty} f(\z) = 0
\]
whenever $j \in I_n$ and $z_k \in \C_+$ are fixed for each $j \ne k$.   Adopting the practice described in Remark \ref{3.2.+},  it suffices that $f \in \B_0^n$ if and only if $f(\z)=0$ whenever $z_j=\infty$ for at least one $j\in I_n$.   We stress that the two sentences above mean precisely the same.

The seminorm $\|\cdot\|_{\B^n_{I_n}}$ is a norm on $\B_0^n$, and we denote this seminorm by $\|\cdot\|_{\B_0^n}$.
Now we show that the norms $\|\cdot\|_{\B_0^n}$ and $\|\cdot\|_{\B^n}$ are equivalent on $\B^n_0$.   In addition, finiteness of $\|f\|_{\B^n_0}$, boundedness of $f$,  and the condition above of vanishing (at infinity, to the right) imply that $f \in \B^n_0$.

\begin{prop}\label{Pr1}
Let $f \in H^\infty(\C_+^n)$, $ \|f\|_{\B^n_0}<\infty$ and assume that
\begin{equation} \label{Hlim}
\lim_{{\rm Re} z_j \to \infty}f(\mathbf{z})=0
\end{equation}
whenever $j\in I_n$ and $z_k$ is fixed for all $k\ne j$.
Then $f \in \mathcal B^n$ and  $\|f\|_{\mathcal B^n}\le 2^n\|f\|_{\mathcal B_0^n}$.
\end{prop}

\begin{proof}
Let $\Omega \in \mathcal {P}_n$, $\Omega\not=I_n$, and $|\O|=k$.
 From the boundedness of $f$, (\ref{Hlim}) and Cauchy's inequality, we infer that 
for all $\Psi\in\P_n$ such that $\O\subseteq\Psi$, any $j\in\O^c$ and $z_k$ fixed for $k\ne j$, 
\[
\lim_{{\rm Re}\, z_j\to\infty}\, \mathcal D_{\Psi}f(\mathbf{z})=0,
\]
and thus
\[
\mathcal D_{\O}f(\mathbf{z})=
\int_{\R_{+}^{n-k}} \mathcal D_{n}f(\mathbf{z}+{\bm \alpha}_{\O^c})\,d{\bm \alpha}_{\O^c}, \qquad \z\in\C_+^n.
\]
By the maximum principle for functions from 
$H^{\infty}(\mathbb C^n)$, we have
\begin{align*}
|\mathcal D_{n}f(\mathbf{z})|&\le
\int_{\R_{+}^{n-k}} 
\sup_{\substack{\lambda_j=z_j,\,j\in \O\\
\Re \lambda_j=\Re z_j, \,j\in \O^c}}
|\mathcal D_{n}f({\bm \lambda}+{\bm \alpha}_{\O^c} )|\,d{\bm \alpha}_{\O^c}\\
&\le
\int_{\R_{+}^{n-k}} 
\sup_{\substack{\lambda_j=z_j,\,j\in\O\\
\Re \lambda_j= 0, \,j\in \O^c}}
|\mathcal D_{n}f({\bm \lambda}+{\bm \alpha}_{\O^c} )|\,d{\bm \alpha}_{\O^c},
\end{align*}
so that
\[
\|f\|_{\B_\O^n} =\sup_{\mathbf{z}\in W_{{\bm \alpha}_{\O} }} \,|\mathcal D_{\O}f(\mathbf{z})|\le
\int_{\R^{n-k}_{+}}\,
\sup_{\mathbf{z}\in W_{{\bm \alpha}_{I_n}}}
|\mathcal D_{n}f(\mathbf{z})|\,d{\bm \alpha}_{\O^c} = \|f\|_{\B^n_0}.
\]
This implies that $f\in\B_0^n$ and 
\[
\|f\|_{\mathcal B^n}=
\sum_{\Omega\in \mathcal{P}_n}
\|f\|_{\mathcal B^n_\Omega}
\le 2^n\|f\|_{\mathcal B_0^n}.  \qedhere
\]
\end{proof}

 Note that $\|\cdot\|_{\B^n_{I_n}}$ is not a norm on $\B^n$, because some functions do not depend on all the $n$ variables. 
  Those functions satisfy $\D_n f = 0$ and hence $\|f\|_{\B^n_{I_n}} = 0$.     However linear combinations $g$ of such functions may depend on all the variables while satisfying $\D_n g = 0$ (see Remark \ref{r2.4}(2) below).   
  Consequently we will decompose $\B^n$ into the direct sum of $2^n$ closed ideals, one for each subset of $I_n$.   
  This was achieved for $\B^1$ in \cite{BGT}, where one summand contained only the constant functions, and also for $\B^2$ in \cite{Kob}, where two summands each contained functions of one of the two variables.   

  Let $f \in \B^n$, and let $\O_f \in \mathcal P_n$ be the subset of $I_n$ representing the variables $z_j$ on which $f$ depends. 
   It is easy to see that they are the variables  for which $D_j f$ is not identically zero.   
   We will use the following terminology:
\begin{enumerate} [(a)]
\item $\O_f$ is the {\it support} of $f$, 
\item $|\O_f|$ is the {\it degree} of $f$,
\item $f$ is {\it elementary} if $f(\z)=0$ whenever $z_j=\infty$ for at least one $j \in\O_f$.
 \end{enumerate}  
 For $\O\in\P_n$, we define subspaces of $\B^n$:
 \begin{align*}
 \B^n_\O &= \left\{ f \in \B^n:  \O_f \subseteq \O \right\} = \left\{f\in\B^n : \text{$\D_{j} f = 0$ for all $j \in \O^c$} \right\}, \\
 \B^n_{\O,0} &= \left\{ f \in \B^n_{\O} : \text{$f(\z)=0$ whenever $z_j=\infty$ for at least one $j\in\O$} \right\}.
 \end{align*} 
  It is easy to see that $\B^n_\O$ and $\B^n_{\O,0}$ are closed subalgebras of $\B^n$.   The spaces $\B^n_{I_n}$ and $\B^n_{I_n,0}$ are the spaces $\B^n$ and $\B^n_0$, respectively, and the space $\B^n_\emptyset$ is the space of constant functions.    Note that
  \begin{enumerate}[(i)]
  \item $f\in\B^n$ is elementary if and only if $f \in \B^n_{\O_f,0}$, and
  \item If $f \in \B^n_0$ and $\D_n f=0$, then $f=0$.
  \end{enumerate}   
 
 For $f \in \B^n$, let 
\[
f_{\rm el}(\z) = \sum_{\O\in\mathcal{P}_n} (-1)^{n-|\Omega|} f_\O(\z_\O),  
\]
where $f_\O(\z_\O)$ is defined in Proposition \ref{3.2.2}.    For example, if $n=2$ then
\[
f_{\rm el}(z_1,z_2) = f(z_1,z_2) - f(z_1,\infty) - f(\infty,z_2) + f(\infty,\infty).
\]
%It is easy to verify that $f_{\rm el} \in \B^n_0$, and the functions $f_\O(\z_\O)$ are all of degree at most $n-1$.    
Then
\[
f = f_{\rm el} + \sum_{\O\in\P_n,\O\ne I_n} (-1)^{n-|\O|-1}f_\O,
\]
where $f_{\rm el}$ is elementary of degree $n$, and each function $f_\O$ for $\O \ne I_n$ is of degree less than $n$.   
The empty set is in $\P_n$, and $f_\emptyset = f_{\emptyset,0}$ is a constant equal to $f(\infty,\dots,\infty)$.

By simple relabelling of the variables, the space $\B_\Omega^n$ can be identified with $\B^k$ and $\B^n_{\Omega,0}$ with  $\B^{k}_{0}$, where $k = |\O|$.   By considering those $f_\O$ for $\O \ne I_n$, and iterating the procedure until we have obtained only elementary functions, we can obtain an {\it elementary decomposition} of $f$ as
\begin{equation} \label{fel}
f = \sum_{\O\in\P_n} f_{\O,0},
\end{equation}
where $f_{\O,0}$ is either identically zero or an elementary function with support $\O$.  

The elementary decomposition of $f \in \B^n$ is unique.   Suppose that $\sum_{\O\in\P_n} f_{\O} = 0$, where $f_\O$ is elementary with support $\O$.   By setting $z_j=\infty$ for all $j \in I_n$, one obtains that the constant function $f_\emptyset = 0$.   By setting $z_j = \infty$ for all $z_j$ except for $j=1$, one obtains that $f_{\{1\}} = 0$.   Continuing in this way one obtains that $f_\O = 0$ for all $\O$.

So far the discussion has been somewhat heuristic.  An alternative description is given in the following formulas, which can be readily verified.

\begin{prop} \label{p2.4}
The Banach algebra $\B^n$ is the direct sum of the closed subalgebras $\B^n_{\O,0}$ for $\O\in\P_n$.    In particular, the following properties hold.
\begin{enumerate}[\rm(i)]
\item  For all $f\in \B^n$,
\[
f(\z) = \sum_{\O\in\P_n} f_{\Omega,0}(\z_\O), \; \text{where} \quad  f_{\O,0}(\z_\O) = \sum_{\Psi\subseteq\O} (-1)^{|\O|-|\Psi|}f_{\Psi}(\z_{\Psi})\in \B^n_{\O,0}.
\]
\item For each $\O\in\P_n$, the map $f\mapsto f_{\O,0}$ for $\O\in\P_n$ is a contraction from $\B^n$ to $\B^{|\O|}$.
\item
For $f\in\B^n$,  $\D_n f=0$ if and only if $f$ is a finite sum of holomorphic functions with degree less than $n$.  In that case,
\[
f(\z) = \sum_{\O\in\mathcal{P}_n,\Omega\ne I_n}  (-1)^{n-|\O|-1} f_\O(\z_\O).
\]
\item \label{p2.4-4} If $g \in H^\infty(\C_+^n)$, $\D_ng=0$, and $g(\z)=0$ whenever $z_k=\infty$ for at least one $k \in I_n$, then $g = 0$.
\end{enumerate}
\end{prop}

\begin{exa} \label{ex2.7}
The space $\LM^n$ of all functions of the form $\LT\mu$, where $\mu$ is a bounded Borel measure on $\R^n_+$ and $\LT\mu$ is its Laplace transform, is a Banach algebra in the norm $\|\LT\mu\|_{HP} = \|\mu\|$, and it is known as the {\it Hille-Phillips algebra}.   It is continuously included in $\B^n$ (see \cite[Section 2.4]{BGT} for the case $n=1$).

Let $f = \LT\mu \in \LM^n$.  Then the elementary decomposition of $f$ is given by $f_{\O,0} = \LT\mu_{\O}$, where $\mu_{\O}$ is the restriction of $\mu$ to the subset
\[
\R_{+,\O} := \prod_{j=1}^n  C_j,    \qquad C_j := \begin{cases} (0,\infty), \; &j\in\O, \\ \{0\}, &j\in\O^c. \end{cases}
\]
\end{exa}

\noindent
\begin{rems} \label{r2.4}
1.  In Propositions \ref{3.2.0}, \ref{3.2.2}, and \ref{p2.4}, instead of letting $\Re z_j$ tend to $\infty$, we could choose to take fixed points $\zeta_j \in \ov\C_+$.   The results remain valid.   Such a change would produce a different elementary decomposition of functions $f\in\B^n$.   However this is not a contradiction of the uniqueness, because the spaces $\B^n_{\O,0}$ would also change.

\noindent
2.   In the paper \cite{Mar}, Marti constructed a functional calculus for two commuting closed operators under certain conditions.  A subsequent short paper by Caradus \cite{Car} showed that the conditions could be put in a better form and pointed out that the construction could then be extended to $n$-tuples of commuting operators in a natural way.   Marti's construction involved a decomposition of a similar type to \cite[Section 3.5]{Kob}, although the side conditions are different.   Caradus did not specify a decomposition but it would undoubtedly have been of similar form to the elementary decomposition in Proposition \ref{p2.4}(i).    Ichinose \cite[p.235]{Ich} also used a similar decomposition in the context of two operators.
\end{rems}

\section{Reproducing formulas and shifts} \label{sect3}

\subsection* {Reproducing formulas}
A reproducing formula for $f \in \B^1$ was given in \cite[Proposition 2.20]{BGT}, and a version for $f \in \B^2$ was given in \cite[Proposition 5.1.1]{Kob}.   In \cite[Proposition 3.7]{BGT3}, some formulas were obtained and the specific formula with $s=1$ can be applied with $g=f'$ if $f\in\B^1$.   We state that formula here.
 
\begin{prop} \label{jj}
Let $g$ be a holomorphic function on $\C_{+}$ such that
\begin{equation*}
\int_{\C_{+}}\frac{|g(z)|\, {\Re}\,z}{|z|^2}\,dS_1(z) <  \infty.
\end{equation*}
Define $Q g$ on $\C_+$ by
\[
(Qg)(z) = -\frac{2}{\pi} \int_{\C_+} g(\l) \frac{\Re\l}{(z+\ov\l)^2} \, dS_1(\l).
\]
Then $Qg$ is holomorphic, and $g = (Qg)'$.
\end{prop}

We will present the corresponding statement for functions of $n$ variables in Proposition \ref{j}. 
We use the following kernel functions:
\begin{equation*}
K(z,\lambda)=-\frac{2}{\pi(z+\lambda)^2}, \; z, \l \in \C_+,
\qquad \mathcal{K}_n(\z,\lb)= \prod_{j=1}^n K(z_j,\l_j), \; \z,\lb\in\C_+^n.
\end{equation*}

Let $g$ be a holomorphic function on $\C_{+}^n$ such that
\begin{equation}\label{CR1}
\int_{\C_{+}^n}|g(\lb)|\prod_{j=1}^n\frac{\Re\l_j}{|\l_j|^2}\,dS_n(\lb)<\infty.
\end{equation}Define the function $\mathcal{Q}_n g$ on $\C_+^n$ by
\[
(\mathcal{Q}_ng)(\z) := \int_{\C_+^n} \mathcal{K}_n(\z,\ov\lb) g(\lb) \prod_{j=1}^n \Re\l_j \,dS_n(\lb), \qquad \z\in\C_+^n.
\]
This function is holomorphic on $\C_+^n$ and the derivatives can be passed through the integral sign, the justification being that the derivatives of the integrand with respect to the $\z$ variables can be estimated locally uniformly by constant multiples of the integrand in (\ref{CR1}).  Thus
\[
(\D_nQ_n g)(\z) = \int_{\C_+^n} \D_n \K_n(\z,\ov\lb) g(\lb)\, dV_n(\lb), \qquad \z\in\C_+^n,
\]
where the derivatives of $\K_n$ are with respect to the $\z$ variables.

We now extend Proposition \ref{jj} to obtain the following reproducing formula for functions on $\C_+^n$.

\begin{prop}\label{j}
Let $g$ be a holomorphic function on $\C_{+}^n$ such that \eqref{CR1} holds.
Then $\mathcal{Q}_ng$ is holomorphic and
\begin{equation}\label{M}
g= \mathcal{D}_n\mathcal{Q}_ng.
\end{equation}
\end{prop}

\begin{proof}
The proof is by induction, and the case $n=1$ is given in Proposition \ref{jj}. 

Let $n\ge1$, and assume that (\ref{M}) holds for $n$ variables.  
We consider now the case of $n+1$ variables $(\z,z_{n+1})$.   
Assume that $g(\z,z_{n+1})$ satisfies \eqref{CR1} for $n+1$ variables.   
For $\z\in\C_+^n$, let $h_{\z}(z_{n+1}) = g(\z,z_{n+1})$.  
By (\ref{CR1}) and Fubini's theorem we obtain that
\[
\int_{\C_{+}}|h_\z(\l)| \frac{\Re{\l}}{|\l|^2}\,dS(\l)<\infty
\]
for almost all $\z\in \C_{+}^n$.
For those $\z$, we can apply Proposition \ref{jj} to the function $h_\z$.    We obtain
\[
g(\z,z_{n+1}) = h_\z(z_{n+1}) = \int_{\C_+} D_{n+1} K(z_{n+1},\l_{n+1}) g(\z,\l_{n+1}) \,dV(\l_{n+1}).
\]
Now let $k_{z_{n+1}}(\z) = g(\z,z_{n+1})$.   By applying Fubini's theorem and the inductive hypothesis to the functions $k_{z_{n+1}}$, we obtain that, for almost all $z_{n+1} \in \C_+$,
\begin{align*}
g(\z,z_{n+1}) &= \int_{\C_+^n} \D_n\K_n(\z,\ov\lb) k_{z_{n+1}}(\lb) \,dV_n(\lb) \\
&= \int_{\C_+^n} \D_n\K_n(\z,\ov\lb) \int_{\C_+} D_{n+1} K(z_{n+1},\l_{n+1}) g(\lb,\l_{n+1}) \,dV(\l_{n+1})\,dV_n(\lb) \\
&= \int_{\C_+^{n+1}} \D_{n+1}\mathcal{K}_{n+1}(\z,z_{n+1},\ov\lb_n,\ov\l_{n+1}) g(\lb,\l_{n+1}) \, dV_{n+1}(\lb,\l_{n+1}) \\
&= (\D_{n+1}Q_{n+1}g )(\z,z_{n+1}).
\end{align*}
This establishes that $g = \D_{n+1}Q_{n+1}g$ almost everywhere.   Since both functions are holomorphic, it follows that they agree everywhere.
\end{proof}

\begin{cor} \label{CorR}
Let $f \in \B^n$.   Then $\D_nf = \D_n\mathcal{Q}_n\D_n f$, and $f_{\rm{el}} = \mathcal{Q}_n\D_n f$, so
 \begin{equation} \label{repf}
f_{\rm{el}}(\z) = \int_{\C_+^n} \mathcal{K}_n(\z,\ov\lb) (\D_nf)(\lb) \, dV_n(\lb), \quad \z\in\C_+^n.
  \end{equation}
 In particular, if $f \in \B^n_0$, then $f = \mathcal{Q}_n \D_nf$.
  \end{cor}

\begin{proof}
Let $g = \D_n f$.   Then
\begin{align*}
\int_{\C_+^n}|g(\z)| \prod_{j=1}^n \frac{\Re z_j}{|z_j|^2} \, dS_n(\z) 
&\le \int_{\R_+^n} (H_{I_n}f)(\aa) \int_{\R^n} \prod_{j=1}^n \frac{\a_j}{\a_j^2+t_j^2} \,d\bb\,d\aa \\
&= \pi^n \|f\|_{\B^n_{I_n}} < \infty.
\end{align*}
Thus $g$ satisfies \eqref{CR1}, so Proposition \ref{j} shows that $g =  \D_n\mathcal{Q}_n g$.   

By the dominated convergence theorem, $(\mathcal{Q}_n\D_n f)(\z)\to0$ if $\Re z_j \to \infty$ for any one $j$.    Moreover,  $f_{\rm{el}} - \mathcal{Q}_n\D_n f$ is in $H^\infty(\C_+^n)$ and in the kernel of $\D_n$ with zero limits as $\Re z_j\to\infty$, and then it follows that $f_{\rm{el}} = \mathcal{Q}_n \D_n f$  (see Proposition \ref{p2.4}(\ref{p2.4-4})).   
\end{proof}

\begin{exa}  \label{exres}
In this example we establish the $\B _k$-norms and other properties of various resolvent functions.

For $z = \a+i\b \in \C_+$ and $\l\in\C_+$, let $r_z(\l) = (\l+z)^{-1}$.   Then $r_z \in \B^1_0$, and 
\[
\|r_z^k\|_{\B_0^1} = \a^{-k}, \quad \|r_z^k\|_{\B^1} =2 \a^{-k}, \quad k \in \N.
\]
%For $\z\in\C_+^n$, $j\in I_n$ and $\l\in\C_+$, let $r_{\l,j,n}(\z) = (z_j + \l)^{-1}$.
For $\z\in\C_+^n$ and $\lb \in \C_+^n$, let
$
r_\z(\lb) := \prod_{j=1}^n r_{z_j}(\l_j).
$
Then $\r_\z \in \B_0^n$, and 
\begin{align*}
\|\r_\z\|_{\B_0^n}  &= \prod_{j=1}^n \|r_{z_j}\|_{\B_0^1} = \prod_{j=1}^n \a_j^{-1}, \quad 
\|\r_\z^k\|_{\B_0^n} = \prod_{j=1}^n \|r_{z_j}^k\|_{\B_0^1} = \prod_{j=1}^n \a_j^{-k},
\end{align*}
where $\a_j = \Re z_j$ and $k \in\N$.   

For $z, \zeta \in \C_+$,
\[
\|r_z-r_\zeta\|_{\B_0^1} = |z-\zeta| \, \|r_z  r_\zeta\|_{\B_0^1} \le |z-\zeta| \|r_z\|_{\B^1}\|\,\|r_\zeta\|_{\B^1} \le \frac{4|z-\zeta|}{\Re z \Re\zeta} \to 0,
\]
as $z\to\zeta$.  So the map $z \mapsto r_z$ is continuous from $\C_+$ to $\B_0^1$.  

For $j \in I_n$, the map $\z \mapsto r_{z_j}$ is continuous from $\C_+^n$ to $\B_0^1$.   Let $r_{z,j}(\lb)  = r_{z}(\l_j)$.  Then $\z \mapsto r_{z,j}$ is continuous from $\C_+^n$ to $\B^n$.  Since $\r_\z = \prod_{j=1}^n r_{z,j}$, the map $\z \mapsto \r_\z$ is continuous from $\C_+^n$ to $\B_0^n$.  

In fact, the map $\z \mapsto r(\z)$ is holomorphic from $\C_+^n$ to $\B_0^n$, but we do not need this.
\end{exa}

\subsection*{Shifts}

We will use shifts of the variables $z_j$ for the construction of the $\B$-calculus for operators.    In \cite{BGT} and \cite{Kob}, we used shifts not only in the horizontal direction but also in the vertical directions.   In this paper, we define the functional calculus by a different route, and we do not need to use the vertical shifts.   So we present two lemmas here about the horizontal shifts on $\B^n$.

First we will establish a variant of the reproducing formula \eqref{repf} for shifted functions in $\B_0^n$.   This result is a special case of  \cite[Lemma 2.18(2)]{BGT} for $n=1$, and \cite[Proposition 5.4.1]{Kob} for $n=2$.   

\begin{lemma}\label{tflip}
If $f \in \mathcal B^n_0,$ then for every $\t\in \R_+^n$,
\begin{equation} \label{fshift}
f(\z+\t) =\int_{\C^n_{+}} \mathcal{K}_n(\z+{\t},\ov{\lb}) (\D_n f)(\lb)\,dV_n(\lb),\qquad \z\in \C^n_{+}.
\end{equation}
\end{lemma}

\begin{proof}
We begin with the case when $n=1$. From Corollary \ref{CorR},
\[
f(z+t)=-\frac{2}{\pi}\int_0^\infty\alpha \int_{\R}\frac{f'(\alpha+t+i\beta)}{(z+\alpha-i\beta)^2}\,d\beta\,d\alpha, \qquad z \in \C_+.
\]
Changing variables in the inner integral, noting that $f'(\a+\l)$ is bounded for $\l\in\C_+$ and a fixed $\a>0$, and using Cauchy's theorem, 
we obtain
\begin{align*}
\int_{\R}\frac{f'(\alpha+t+i\beta)}{(z+\alpha-i\beta)^2}\,d\beta
&=-i\int_{t+i\R}\frac{f'(\alpha+\l)}{(z+\alpha+t-\l)^2}\,d\l\\
&=-i\int_{i\R}\frac{f'(\alpha+\l)}{(z+\alpha+t-\l)^2}\,d\l
=\int_{\R}\frac{f'(\alpha+i\beta)}{(z+\alpha+t-i\beta)^2}\,d\beta.
\end{align*}
So, we obtain the formula \eqref{fshift} for $n=1$:
\[
 f(z+t)=-\frac{2}{\pi}\int_0^\infty\alpha \int_{\R}\frac{f'(\alpha+i\beta)}{(z+t+\alpha-i\beta)^2}\,d\beta\,d\alpha, \qquad z\in \C_{+}, \, t\ge0.
\]

For $n>1$, we may consider the individual variables one at a time, applying the same argument as above to show that
\begin{multline*}
 \hskip40pt \int_{\C^n_{+}} \mathcal{K}_n(\z+{\t^{j-1}},\ov{\lb}) (\D_n f)(\lb+\t_{j})\,dV_n(\lb)  \\
=\int_{\C^n_{+}} \mathcal{K}_n(\z+{\t^{j}},\ov{\lb}) (\D_n f)(\lb+\t_{j+1})\,dV_n(\lb), \hskip40pt
\end{multline*}
where 
\[
\t^{j} = (t_1,\dots,t_{j},0,0,\dots,0), \; \t_{j} = (0,\dots,0,t_{j}, t_{j+1},\dots, t_n), \quad j\in I_n.  \qedhere
\]
\end{proof}

Next we show that horizontal shifts are strongly continuous on $\B^n$.   For $f \in \mathcal B^n$, $\t=(t_1, \dots, t_n) \in \R_+^n$,
and $\z \in \mathbb C_+^n$, let
\begin{equation} \label{Tdef}
(T(\t)f)({\z}):=f(\z+\t).
\end{equation}
It is easy to see that $T(\t)$ is a contraction on $(\B^n,\|\cdot\|_{\B^n})$.
The following lemma was proved in \cite[Lemma 2.6]{BGT} for $n=1$ and in \cite[Proposition 4.2.2]{Kob} for $n=2$.

\begin{lemma}  \label{C0}
The map $T : \R_+^n \to L(\B^n)$ is strongly continuous.
\end{lemma}

\begin{proof}
Let
\[
(T_j(t)f)({\z}):= f(z_1, \dots, z_j+t, \dots, z_n), \qquad j \in I_n,\, t\ge0.
\]
Then
\[
T(\t)= T_1(t_1) \dots T_n(t_n).
\]
It suffices to show that the semigroups $(T_j(t))_{t\ge0}$
 are continuous in the strong operator topology on $L(\B^n)$.   
 Without loss of generality, we will take $j=1$, and we will write $\z=(z_1,\z') \in \C_+^n$.

Let  $f\in\B^n$ and $\Omega\in\mathcal{P}_n$ with $k = |\O|$.   We need to show that 
\begin{equation} \label{C0B}
\lim_{t\to0} \|T_1(t)f - f\|_{\B^n_\O} = 0.
\end{equation}
There are two cases: (1) $1 \in \O$, (2) $1 \notin \O$.

\noindent
{\it Case 1: $1 \in \O$}.
Let $\aa_\O = (\aa_j)_{j\in\O} \in (0,\infty)^k$, and let $\z \in W_{\aa_\O}$.  
By an extended version of Cauchy's inequality,
\begin{equation} \label{Poi}
\left| (D_1\D_\O f)(\z) \right| \le \frac{2^{-(k-1)}\|f\|_\infty} { \a_1 \prod_{j\in\O} \a_j}.
\end{equation}
This estimate is valid for $\Re z_1 > \a_1$, by the maximum principle.    Hence
\begin{align*}
\left| (\D_\O f)(z_1+t,\z') - (\D_\O f)(z_1,\z') \right| \le  \int_0^t \left| (D_1\D_\O f)(z_1+s,\z') \right|\,ds 
&\le \frac{2^{-(k-1)} t \|f\|_\infty} {\a_1 \prod_{j\in\O} \a_j}.
\end{align*}
Taking the supremum over all possible choices of $\z \in W_{\aa_\O}$, this establishes that
\[
H_\O(T_1(t)f-f)(\aa_\O) \le \frac{2^{-(k-1)}t \|f\|_\infty} {\a_1 \prod_{j\in\O} \a_j}  \to 0,   \quad t\to0.
\]
In addition,
\begin{equation}  \label{HO}
H_\O(T_1(t)f-f)(\aa_\O) \le (H_\O T_1(t)f)(\aa_\O) + H_\O f(\aa_\O)  \le 2 H_\O f(\aa_\O),
\end{equation}
by the maximum principle.   Hence, the dominated convergence theorem implies that \eqref{C0B} holds.

\noindent
{\it Case 2: $1 \notin \O$}.  
 Let $\O'=\O \cup \{1\}$, $\aa_\O = (\aa_j)_{j\in\O} \in (0,\infty)^k$, and $\z \in W_{\aa_\O}$.  Then $\z_{\O'}   = (z_1,\z_{\O})$.   The definition of $H_{\O'}$ and its monotonicity show that
 \begin{multline*}
\left| (\D_\O f)(z_1+t,\z') - (D_\O f)(z_1,\z') \right| \le  \int_0^t \left| (\D_{\O'} f)(z_1+s,\z') \right|\,ds \\
\le  \int_0^t  (H_{\O'}f)(\a_1+s,\aa_{\O})\, ds \le  \int_0^t  (H_{\O'}f)(s,\aa_{\O}) \, ds.
\end{multline*}
The final expression depends only on $f$, $t$ and $\aa_\O$, so we may take the supremum over all $\z\in W_{\aa_\O}$ and infer that
\begin{equation} \label{Hc2}
H_\O(T_1(t)f-f)(\aa_\Omega)  \le \int_0^t  (H_{\O'}f)(s,\aa_{\O}) \, ds.
\end{equation}
Since $\|f\|_{\B^n_{\O'}}$ is finite, $H_{\O'}f$ 
is integrable over $\R_+^{k+1}$.   Fubini's theorem implies that 
\[\a_1 \mapsto \int_{\R_+^k} (H_{\O'}f) (\a_1,\aa_\O) \,d\aa_\O  \]
is integrable over $\R_+$, and then \eqref{Hc2} and Fubini's theorem imply that
\begin{align*}
\|T_1(t)f - f\|_{\B^n_\O} &=  \int_{\R_+^k} H_\O(T_1(t)f-f)(\aa_\Omega)  \,d\aa_\O \\
&\le \int_0^t \int_{\R_+^k} (H_{\O'}f) (\a_1,\aa_\O) \,d\aa_\O \,d\a_1 
\to 0, \qquad t\to 0.  \qedhere
\end{align*}
\end{proof}

\begin{rems}  \label{3.5}

\noindent
1.  Case 2 of the proof above can be applied when $\O$ is empty, to show that the left shifts are strongly continuous in the $H^\infty$-norm, and hence that $f$ is uniformly continuous, as shown in Proposition \ref{3.2.4}.  

\noindent
2.
For $n=2$, the proof above differs from \cite[Proposition 4.2.2]{Kob}, where the proof relies on switching to an equivalent norm on $\B^2$.   That approach can also be applied for $n>2$.   It simplifies the proof of case 2 above, provided one accepts the equivalence of norms.

\noindent
3.
It was shown in \cite[Proposition 4.6.2]{BGT} that the shift semigroup on $\B^1$ extends to a bounded holomorphic semigroup of angle $\pi/2$.
\end{rems}

\subsection*{Sums of variables}

In this section, we consider how very simple changes of  the variables $z_j$ preserve $\B^n$.   As noted in \cite[Lemma 2.6(5)]{BGT} a trivial scaling of variables shows that if $f\in\B^1$, $b >0$  and $g(z) = f(bz)$,  then $g \in \B^1$, and $\|g\|_{\B^1} = \|f\|_{\B^1}$, and the same applies for functions in $\B^n$ with scaling of some or all variables.

In addition, the family of algebras $\B^n$ is preserved by summing some or all of the variables $z_j$.   In order to avoid heavy formulas, we present this only in the case where all the variables are summed.   For $n=2$, the result was proved in \cite[Lemma 3.8.1]{Kob}.  

\begin{prop} \label{sum}
Let $f \in \B^1$, $n\in\N$, and define $g$ on $\C_+^n$ by
\begin{equation} \label{sumb}
g(\z) = f(z_1 + \dots + z_n), \qquad  \z\in\C_+^n.
\end{equation}
Then $g \in \B^n$, and the map $f \mapsto g$ from $\B^1$ to $\B^n$ is bounded.
\end{prop}

\begin{proof}
We may assume that $f\in\B^1_0$.   Clearly $g$ is holomorphic on $\C_+^n$ and $g(\z)\to0$ as $\Re z_j\to\infty$ for some $j \in I_n$.
By Proposition \ref{Pr1}, it suffices to show that $\|g\|_{\B^n_0} < \infty$.      We have
\begin{equation*}
\|g\|_{\B^n_0} = \int_{\R^n_+} \sup_{\b\in\R} \left|  f^{(n)}(\a_1+\dots+\a_n + i\b)\right| \, d\aa.
\end{equation*}
The change of variables
\[
t_k  = \sum_{j=1}^k \a_j, \qquad k \in I_n,
\]
establishes that
\begin{align*} 
\|g\|_{\B^n_0} &= \int_{S_n}  \sup_{\b\in\R} \left|  f^{(n)}(t_n+i\b)\right| \, d\t,
\end{align*}
where $S_n = \{ (t_k)_{k\in I_n} \in \R_+^n : 0 \le t_1 \le t_2 \le \dots \le t_n\}$.   Integrating with respect to $t_1, \dots, t_{n-1}$ and replacing $t_n$ by $t$ shows that
\[
\|g\|_{\B^n_0} = \int_0^\infty \frac{t^{n-1}}{(n-1)!} \sup_{\b\in\R}  \left|f^{(n)}(t+i\b)\right| \, dt.
\]
A higher-order version of Cauchy's inequality gives
\[
t^{n-1} \sup_{y\in\R}|f^{(n)}(t+i\b)| \le 2^{n-1} C_{n-1} \sup_{\b\in\R} |f'({t/2+i\b})|,
\]
where $C_{n-1}$ is as in \cite[Lemma 2.1(4)]{BGT}.   Hence
\[
\|g\|_{\B_0^n} \le \frac{2^{n-1} C_{n-1}}{(n-1)!} \int_0^\infty \sup_{\b\in\R} |f'(t/2+i\b)| \,dt = \frac{2^{n}C_{n-1}}{(n-1)!}  \|f\|_{\B_0^1}.
\]
This establishes that $g \in \B_0^n$, and also that the map $f \mapsto g$ is bounded from $\B^1$ to $\B^n$.
\end{proof}

It is clear that if $g$ as defined in \eqref{sumb} is in $\B^n$, then $f \in \B^1$, since $f(z) = g(z,0,0,\dots,0)$.

\section{The $\B^n$-calculus: set-up} \label{sect4}

In the remainder of this paper, $\A := (A_1,\dots,A_n)$ will be an $n$-tuple of commuting operators on a Banach space $X$.   The operators $A_j$ commute in the sense that  the resolvent set of each operator 
 $A_j$ is non-empty and their resolvents commute with each other.    Moreover, we will always assume that the spectrum of each operator satisfies
\begin{equation} \label{spr}
\sigma(A_j) \subseteq \ov\C_+, \quad j \in I_n.
\end{equation}

Assuming that \eqref{spr} holds, let $\O\in\P_n$.   We will say that $\A$ satisfies the {\it {\rm (GSF$_\O$)} condition} if the following holds for all $x\in X$, and $x^* \in X^*$:
\[
\sup_{\a_j>0, j\in\Omega} \int_{\R^{|\O|}} \big|\langle \prod_{j\in\O} \a_j (A_j + \a_j-i\b_j)^{-2} x,x^* \rangle \big| \, d\bb < \infty.  \tag{GSF$_{\O}$}
\] 
By the Closed Graph theorem, if (GSF$_\O$) holds, there is a constant $\gamma_{\A_\O}$ such that 
\begin{equation}  \label{GSF2}
\int_{\R^{|\O|}} \big| \langle \prod_{j\in\O} \a_j (A_j + \a_j-i\b_j)^{-2} x,x^* \rangle \big| \, d\bb \le \gamma_{\A_\O} \|x\|\,\|x^*\|
\end{equation}
for all $x \in X$ and $x^* \in X^*$. 
If $\Omega$ is empty, then the condition is deemed to be satisfied.  
For $\O=\{j\}$, the condition says that $A_j$ satisfies the (GSF) condition for the single operator $A_j$ as in \cite[Section 4.1]{BGT}.    
 If $\O = I_n$, we will say that $\A$ satisfies the (GSF$_n$) condition and we will write $\gamma_\A$ instead of $\gamma_{I_n}$. 

We will say that $\A$ satisfies the {\it full {\rm(GSF)} condition} if $\A$ satisfies (GSF$_\O$) for all $\O \in \P_n$.    For $n=2$, the full (GSF) condition agrees with \cite[Section 6.1]{Kob}.

\begin{defn} \label{deffc}
A {\it $\B^n$-calculus} for $\A$ is a bounded algebra homomorphism $\Phi : \B^n \to L(X)$ such that $\Phi(r_{\l,j}) = (A_j+\l)^{-1}$ for all $j\in I_n$ and $\l \in\C_+$. 
\end{defn} 

There is a similar definition of a $\B_0^n$-caclulus.   In this and the next section, we will show that the full (GSF) condition is a sufficient condition for $\A$ to have a $\B^n$-calculus, but we first show that the full (GSF) condition is necessary.  This was proved in \cite[Theorem 6.1]{BGT2} for $n=1$ and in \cite[Theorem 6.3.1]{Kob} for $n=2$.  The proof below is a variant of those proofs.

We will write
$
\r_\z(\A) := \prod_{j=1}^n (A_j+z_j)^{-1}.
$

\begin{prop} \label{nec}
Let $\A$ be an $n$-tuple of commuting operators on a Banach space $X$ satisfying \eqref{spr}, and assume that there is a $\B_0^n$-calculus $\Phi$ for $\A$.   Then $\A$ satisfies the {\rm (GSF$_n$)} condition.
\end{prop}

\begin{proof}
Let $\varphi : \R^n \to \C$ be a continuous function with compact support, and $\aa\in(0,\infty)^n$.   As shown in Example \ref{exres}, the map $\bb \mapsto \r^2_{\aa-i\bb}$ is continuous from $\R^n$ to $\B_0^n$ and $\|\r^2_{\aa-i\bb}\|_{\B_0^n} \le 4^n \prod_{j=1}^n \a_j^{-1}$.   Hence
\[
G_{\aa,\varphi} := \int_{\R^n} \varphi(\bb)  \r_{\aa-i\bb}^2 \, d\bb,
\]
exists as a Bochner integral with values in $\B_0^n$.   Since $\Phi$ is a bounded linear operator,
\[
\Phi(G_{\aa,\varphi}) =\int_{\R^n} \varphi(\bb)  \r_{\aa-i\bb}(\A)^2 \, d\bb
\]
and hence, for unit vectors $x \in X$ and $x^* \in X^*$, 
\[
\prod_{j=1}^n \a_j \left| \int_{\R^n} \varphi(\bb)  \langle \r_{\aa-i\bb}(\A)^2 x, x^* \rangle   \, d\bb \right|
\le 4^n \|\varphi\|_\infty \|\Phi\|.
\]
Since this holds for all continuous functions $\varphi$ with compact support, it follows that (GSF$_n$) holds.
\end{proof}

\begin{cor}  \label{Bne}
Let $\A$ be an $n$-tuple of comuuting operators on a Banach space $X$ satisfying \eqref{spr}, and assume that there is a bounded algebra homomorphism $\Phi : \B^n\to L(X)$ such that $\Phi(\r_\z) = \r_\z(\A)$ for all $\z \in \C_+^n$.   Then $\A$ satisfies the full {\rm(GSF)} condition.
\end{cor}

\begin{proof}
For $\O\in\P_n$, we may deduce that (GSF$_\O$) holds, by applying Proposition \ref{nec} to $\B^\O_0$ and the restriction of $\Phi$ to that algebra.
\end{proof}

Now we give two classes of operators where $\A$ satisfies the full (GSF) condition for all values of $n$. 
 Proofs were given in \cite{BGT} for $n=1$ and in \cite{Kob} for $n=2$, and the general proofs are minor variants of those cases.
 
 \begin{exas}  \label{ex4.1}
1.   Let $-A_j$ be the generator of a bounded $C_0$-semigroup $(e^{-tA_j})_{t\ge0}$ on a Hilbert space $X$, for each $j\in I_n$, and let $K_j = \sup_{t\ge0} \|e^{-tA_j}\|$.   Assume that the semigroups commute with each other (or equivalently, the resolvents of  $A_j$ commute with each other).   Then $\A$ satisfies the full (GSF) condition, with $\gamma_\A = 2^n \prod_{j=1}^n K_j^2$.   This follows from applications of Plancherel's theorem and the Cauchy-Schwarz inequality.   See the proofs for $n=1$ in \cite[Examples 4.1]{BGT} and for $n=2$ in \cite[Example 6.1.1]{Kob}.  

\noindent 2.   Let $A_j \, (j\in I_n)$ be sectorial operators of angle less than $\pi/2$ on a Banach space $X$, and assume that the resolvents commute with each other (equivalently, $-A_j$ is the generator of a bounded holomorphic $C_0$-semigroup $(e^{-tA_j})_{t\ge0}$ on a Banach space $X$, for $j\in I_n$, and the semigroups commute with each other).  Let $M_j = \sup_{z\in\C_+} \|z(A_j+z)^{-1}\|$.  Then $\A$ satisfies the full (GSF) condition, with $\gamma_\A =2^n \prod_{j=1}^n M_j^2$.   This follows from standard resolvent estimates for sectorial operators.  See the proofs for $n=1$ in \cite[Section 4.2]{BGT} and for $n=2$ in \cite[Example 6.1.2]{Kob}.   
\end{exas}

Now, assume that $\A$ satisfies the (GSF$_n$) condition, and let $f \in \B_0^n$.   Recall from Corollary \ref{CorR} that  the following reproducing formula holds:
\begin{equation}\label{A0}
f(\z)
=\int_{\C_{+}^n} \mathcal{K}_n(\z,\ov{\lb})(\D_n f)(\lb)\,dV_n(\lb),       \qquad \z \in \C_+^n.
\end{equation}
We define $f(\A)$ by the corresponding formula:
\begin{equation}\label{OperA1}
\langle f(\A)x,x^* \rangle:=\int_{\C_{+}^n} \langle \mathcal{K}_n(\A,\ov{\lb})x,x^* \rangle \,(\D_n f)(\lb)\, dV_n(\lb),
\end{equation}
where
\begin{equation} \label{defK}
\K_n(\A,\ov{\lb}) :=  \left(- \frac{2}{\pi} \right)^n \r_{\ov\lb}(\A)^2 =     \left(- \frac{2}{\pi} \right)^n \prod_{j=1}^n (A_j+\ov\l_j)^{-2}, \qquad \lb \in \C_+^n.
\end{equation} 
By the (GSF$_n$) condition, the integral is absolutely convergent, and the formula \eqref{OperA1} defines $f(\A)$ as a map in $L(X,X^{**})$, with norm at most $\gamma_\A \|f\|_{\B^n_0}$.   In order to obtain a $\B^n$-calculus for $\A$, we need to extend the definition of $f(\A)$ to functions $f \in \B^n$, show that $f(\A)$ maps $X$ into $X$ for all $f\in \B^n$, and show that the map $f \mapsto f(\A)$ is an algebra homomorphism from $\B^n$ to $L(X)$. 

If the operators $A_j$ are the negative generators of $n$ commuting $C_0$-semigroups, and 
$f =\LT\mu$ for some bounded measure $\mu$ on $\R_+^n$, then $f(\A)$ as defined above agrees with the Hille-Phillips calculus for $\LM^n$.   This was shown in \cite[Lemma 4.2]{BGT} for $n=1$, and in  \cite[Lemma 6.1.4]{Kob} for $n=2$.   The proof for $n\ge3$ is similar, using Example  \ref{ex2.7} and higher-order integrals.  A consequence of this is that the definition of $\K_n(\A,\ov\lb)$ in \eqref{defK} agrees with the definition obtained by putting $f(\z) = \K_n(\z,\ov\lb)$ in \eqref{OperA1}, for a fixed $\lb$.

Now we assume that $\A$ satisfies the full (GSF) condition. 
 If $f$ is an elementary function in $\B^n$, with order $k$ and support $\O$,  we can define $f(\A)$ by applying the $\B_0^k$-calculus for $\A_\O := (A_j)_{j\in\O}$  to the function $f_\O \in \B_0^k$, so
\begin{align}\label{def2}
\langle f(\mathcal{A})x,x^* \rangle &:=\int_{\C_{+}^k} \langle \mathcal{K}_k (\mathcal{A}_{\O},\ov{\lb})x,x^*\rangle\, 
(\mathcal D_k f_\O)(\lb)\,dV_k(\lb)
\end{align}
for all $x \in X$ and $x^* \in X^*$.   In particular, if $\O=\{j\}$ for some $j\in I_n$, $g\in\B^1$, and we define $f(\z) = g(z_j)$, then we recover the $\B^1$-calculus for the operator $A_j$.   Thus this definition establishes that $f(\A) = (A_j+\l)^{-1}$ if $f(\z) = r_\l(z_j)$.

Now let $f$ be an arbitrary function in $\B^n$.  To define $f(\A)$ we use the elementary decomposition of $f$ as in Proposition \ref{p2.4}(i).   
We define
\begin{equation} \label{deffA}
f(\A) = \sum_{\O\in\P_n} f_{\O,0}(\A_\O),  
\end{equation}
Then we define a bounded linear map from $\Phi_\A:\B^n \to L(X,X^{**})$ by
\begin{equation} \label{defPA}
\Phi_\A(f) = f(\A).
\end{equation}
Then $\|\Phi_\A\| \le \gamma_\A := \max_{\O\in\P_n} \gamma_{\A_\O}$, where $\gamma_{\A_\O}$ is as in \eqref{GSF2}.    The norm of the restriction of $\Phi_\A$ to $\B_0^n$ is the minimal value of $\gamma_{\A_{n}}$, and we will denote it by $\gamma_{\A,0}$.
 
In the next section, we will prove that $\Phi_\A$ is an algebra homomorphism from $\B^n$ to $L(X)$.   
It will suffice to show that $f(\A) \in L(X)$ and $(f g)(\A) = f(\A)g(\A)$ for all elementary functions $f$ and $g$ in $\B^n$.    
These properties were proved in \cite{BGT} for $n=1$ and in \cite{Kob} for $n=2$, by a rather complicated method, and the method used in Section \ref{sect5} is simpler, and illuminating even in the case $n=1$.    

Before moving on to the main proof, we present two lemmas.   The first one may help when considering \eqref{deffA}. 

 \begin{lemma}
Assume that $\A = (A_1,\dots,A_n)$ satisfies the full {\rm(GSF)} condition, and let $\A_{n-1} =(A_1,\dots,A_{n-1})$,  $x\in X$ and $x^* \in X^*$.   
\begin{enumerate} [\rm1.]
\item For $f \in \B^n$ and $\O\in\P_n$ with $|\O|=k$, 
\begin{equation}\label{OperA2}
\langle f_{\O,0}(\A)x,x^* \rangle = \int_{\C_{+}^k} \langle \mathcal{K}_k(\A_\O,\ov{\lb})x,x^* \rangle \,(\D_\O f_\O)(\lb)\, dV_k(\lb).
\end{equation}
\item Let $f\in\B^n_0$.   Then $ (D_n f)(\cdot,\lambda_n)\in \mathcal{B}^{n-1}_0$ for all $\lambda_n\in \C_{+}$, and
\[
\langle f(\A)x, x^* \rangle =\int_{\C_{+}} \langle K(A_n,\ov{\lambda}_n)(D_n f)(\A_{n-1},\lambda_n)
x,x^{*}\rangle \, dV(\lambda_n).
\]
\end{enumerate}
\end{lemma}

\begin{proof}  The first statement is proved in Corollary \ref{CorR}.

 Proposition \ref{3.2.0}, with $\O=I_{n-1}$ and $\Psi=\{n\}$, establishes that 
 $(D_n f)(\cdot,\lambda_n)\in \mathcal{B}^{n-1}_0$ for all $\l_n\in\C_+$.   
Fubini's theorem shows that
\begin{align*}
\hskip-20pt\lefteqn{\langle f(\mathcal{A})x,x^* \rangle} \ \\
 %&=\int_{\C_{+}^n} \langle \mathcal{K}(\mathcal{A},\ov{\lb})x,x^*\rangle\, 
%(\D_n f)(\lb)\,dV_n(\lb) \\
&=\int_{\C_{+}^n}
\langle \K_{n-1}(\A_{n-1},\ov{\lb}_{n-1})K(A_n,\ov{\lambda}_n)
x,x^* \rangle \,(\D_n f)(\lb_{n-1},\l_n)\,dV_n(\lb_{n-1},\l_n) \\
&=\int_{\C_{+}}
\langle K(A_n,\ov{\lambda}_n)(D_n f)(\A_{n-1},\lambda_n)
x,x^{*}\rangle \, dV(\lambda_n), \notag
\end{align*}
for all $x \in X$ and $x^* \in X^*$.
\end{proof}

In order to justify some applications of Fubini's theorem, and differentiation through integral signs, in the next section, we will apply estimates for $n$-tuples of operators of the form $\A+\t := (A_j+t_j)_{j\in I_n}$,  where $t_j>0$, and we will use the following fact.   For $n=1$, it is a special case of \cite[Corollary 4.5]{BGT}.    Here $(T(\t))_{\t\in\R_+^n}$ are the shifts as defined in (\ref{Tdef}).

\begin{lemma}  \label{shift3}
Assume that $\A$ satisfies the {\rm(GSF$_n$)} condition.  
Let $f \in \B_0^n$ and $\t \in \R_+^n$.   Then
\[
f(\mathcal A +\t)=(T(\t) f)(\mathcal A).
\] 
\end{lemma}

This can easily be proved by the same method as in Lemma \ref{tflip}, with $\K_n(\z,\ov\lb)$ being replaced by $\langle \K_n(\A, \ov\lb)x,x^* \rangle$, $\K_n(\z+\t,\ov\lb)$ by $\langle \K_n(\A+\t, \ov\lb)x,x^* \rangle$,  and $(\D_nf)(\lb)$ by $(\D_n f)(\lb+\t)$.  

\section {The $\mathcal{B}^n$-calculus: developments}  \label{sect5}

In this section, we assume that $\mathcal{A}: =(A_1,\dots,A_n)$
is an $n$-tuple of commuting operators which satisfy the full {\rm(GSF)} condition on a Banach space $X$,
and we denote by $\Phi_\A$ the map $f \mapsto f(\A)$ from $\B^n$ to $L(X,X^{**})$,
 as defined in (\ref{def2}), (\ref{deffA}) and (\ref{defPA}).   
We will establish three properties:  
\begin{itemize}
\item the operator $f(\A)$  maps $X$ into $X$,  
\item the operator $\Phi_\A$ is an algebra homomorphism,
\item $\Phi_\A$ is the unique $\B^n$-calculus for $\A$.  
\end{itemize}
The first two properties establish that $\Phi_\A$ is a $\B^n$-calculus for $\A$.   
For $n=1$ all three properties have been proved, the first two in \cite[Theorem 4.4]{BGT}, and the third in \cite[Theorem 6.6]{BGT2}.   For $n=2$, they have been proved in \cite{Kob}.
The proofs in those papers can probably be adapted to cover the same properties for arbitrary $n\in\N$, 
but we will present more direct proofs of all three properties even for $n=1$.  

We proceed with three lemmas which combine to prove the three properties, 
but the three proofs are inter-related.
 We start by considering two non-zero elementary functions $f$ and $g$ in $\B^n$.   We denote their supports by $\Omega_f$ and $\Omega_g$ with cardinalities $k$ and $m$, respectively.   Thus $f$ and $g$ belong to the spaces $\B_0^{\O_f}$ and $\B_0^{\O_g}$, respectively.   Unlike the case when $n=1$, the two supports may be different.  As it turns out, this does not cause a serious complication.

\begin{lemma}\label{FA2n}
Let $\mathcal{A}=(A_1,\dots,A_n)$ be an $n$-tuple of operators which satisfy the full {\rm(GSF)} condition on a Banach space $X$.  Let $f$ and $g$ be elementary functions in $\B^n$, and assume that $g(\A+\t) \in L(X)$ for $\t = (t,\dots,t)$, $t\ge0$.   Then $f g$ is an elementary function in $\B^n$ and
\begin{equation}\label{AsF2}
(f g)(\mathcal{A})=f(\mathcal{A})g(\mathcal{A}).
\end{equation}
In particular,
\[
(f \, \bm r_\z)(\A) = f(\A) \r_\z(\A), \qquad \z\in\C_+^n.
\]
\end{lemma}

\begin{proof}
If $f$ or $g$ is the zero function, the result holds.   So we assume that $f$ and $g$ are not zero, and their supports have cardinalities $k$ and $m$, respectively.  It is clear that $fg$ is an elementary function in $\B^n$, and its support is $\O := \O_f \cup \O_g$.   Consequently we can also assume that $\O = I_n$.
 
Let $\t=(t,\dots,t)\in (0,\infty)^n$, 
$f_{\t}(\z)=f(\z+\t)$ and $g_{\t}(\z)=g(\z+\t)$.
Applying (\ref{OperA1}) to $f$ and $g$, we obtain, for $\z\in\C_+^n$,
\begin{align*}
&(f_{\t} g_{\t})(\z)\\
&=\int_{\C_{+}^m}\int_{\C_{+}^k}
\left[\mathcal{K}_{k}(\z_{\O_f}+\t_{\O_f},\ov{\bm\l})\mathcal{K}_{m}(\z_{\O_g}+{\t}_{\O_g},\ov{\bm\nu})\right]
(\mathcal{D}_{\O_f}f)(\bm\l)(\mathcal{D}_{\O_g}g)(\bm\nu) \, dV_k(\bm\l)
\,dV_m(\bm\nu).
\end{align*}
Differentiation through the integral (which is easily justified) gives
\begin{align*}
 \mathcal{D}_n(f_{\t} g_{\t})(\z)
&=\int_{\C_{+}^m}\int_{\C_{+}^k}
\mathcal{D}_n \left[\mathcal{K}_{k}(\z_{\O_f}+\t_{\O_f},\ov{\bm\lambda})\mathcal{K}_{m}(\z_{\O_g}+\t_{\O_g},\ov{\bm\nu})\right] \\
&\null\hskip90pt \times (\mathcal{D}_{\O_f}f)(\bm\lambda)(\mathcal{D}_{\O_g}g)(\bm\nu) \,dV_k(\bm\lambda)
\,dV_m(\bm\nu).
\end{align*}
Combining \eqref{def2},  
Lemma \ref{tflip}, and Fubini's theorem (see the Appendix for justification), 
we infer that
\begin{align}\label{FF12}
\lefteqn{\null\hskip15pt \langle (f_{\t} g_{\t})(\mathcal{A})x,x^{*}\rangle}\\
&=\int_{\C_{+}^n} \langle \mathcal{K}_{n}(\mathcal{A},\ov{\z})x,x^{*}\rangle 
\int_{\C_{+}^m}\int_{\C_{+}^k} \mathcal{D}_{n}\left(\mathcal{K}_{k}(\z_{\O_f}+\t_{\O_f},\ov{\bm\lambda})
\mathcal{K}_{m}(\z_{\O_g}+\t_{\O_g},\ov{\bm\nu})\right) \notag \\
& \hskip130pt \times (\mathcal{D}_{\O_f}f)(\bm\lambda)(\mathcal{D}_{\O_g}g)(\bm\nu)\,
\,dV_k(\bm\lambda) \,dV_m(\bm\nu) \,
dV_{n}(\z)\notag \\
&=\int_{\C_{+}^m}\int_{\C_{+}^k} (\mathcal{D}_{\O_f}f)(\bm\lambda)
(\mathcal{D}_{\O_g}g)(\bm\nu)\,
R_{\t}(\bm\lambda,\bm\nu)\, dV_k(\bm\lambda)\,dV_m(\bm\nu),\notag
\end{align}
where
\[
R_\t({\bm\lambda},\bm\nu)=\int_{\C_{+}^{n}}\langle \mathcal{K}_{n}(\mathcal{A},\ov{\z})x,x^{*}\rangle  
\mathcal{D}_{n}\left(\mathcal{K}_{k}(\z_{\O_f}+\t_{\O_f},\ov{\bm\lambda})
\mathcal{K}_{m}(\z_{\O_g}+\t_{\O_g},\ov{\bm\nu})\right) dV_{n}(\z).
\]

For all  $\bm\lambda\in \C^k_{+}$ and $\bm\nu\in \C_{+}^m$,   
\[
\mathcal{K}_{k}(\cdot+\t_{\O_f},\ov{\bm\lambda})\in  \mathcal{LM}^k\cap \mathcal{B}_0^k,\qquad
\mathcal{K}_{m}(\cdot+\t_{\O_g},\ov{\bm\nu})\in  \mathcal{LM}^m\cap \mathcal{B}_0^m,
\]
as functions of $\z_{\O_f}$ and $\z_{\O_g}$, respectively.   
 Both these functions, and their product, are in $\mathcal{LM}^n$.   
Using \eqref{def2} and the fact that the Hille-Phillips calculus is an algebra homomorphism, we see that
\begin{align}
R_{\t}(\bm\lambda,\bm\nu) 
&=\langle [\mathcal{K}_{k}(\cdot+\t_{\O_f},\ov{\bm\lambda})
\mathcal{K}_{m}(\cdot+\t_{\O_g},\ov{\bm\nu})](\mathcal{A})x,x^{*}\rangle\\
&=\langle \mathcal{K}_{k}(\mathcal{A}_{\O_f}+\t_{\O_f},\ov{\bm\lambda})
\mathcal{K}_{m}(\mathcal{A}_{\O_g}+\t_{\O_g},\ov{\bm\nu}) x,x^{*}\rangle.\notag
\end{align}
Thus, by (\ref{FF12}),  \eqref{def2}, and Lemma \ref{shift3},
\begin{align}
\lefteqn{\null\hskip10pt \langle (f_{\t} g_{\t})(\mathcal{A})x,x^{*}\rangle}\\
&=\int_{\C_{+}^m}\int_{\C_{+}^k} (\mathcal{D}_{\O_f}f)(\bm\lambda)
(\mathcal{D}_{\O_g}g)(\bm\nu) \nonumber\\
&\null \hskip70pt \langle \mathcal{K}_{k}(\mathcal{A}_{\O_f}+\t_{\O_f},\ov{\bm\lambda})
\mathcal{K}_{m}(\mathcal{A}_{\O_g}+\t_{\O_g},\ov{\bm\nu}) x,x^{*}\rangle
\,dV_k(\bm\lambda)\,dV_m(\bm\nu)  \nonumber\\
&= \int_{\C_+^k} \langle (\D_{\O_f}f)(\lb) \langle g(\A+ \t) \K_k(\A_{\O_f} + \t_{\O_f},\ov\lb)x,x^* \rangle \, dV_k(\lb) \nonumber\\
&= \int_{\C_+^k} \langle (\D_{\O_f}f)(\lb) \langle \K_k(\A_{\O_f} + \t_{\O_f},\ov\lb) g(\A+\t) x,x^* \rangle \, dV_k(\lb) \nonumber\\
&=\langle f(\mathcal{A}+\t)g(\mathcal{A}+\t)x,x^{*}\rangle \nonumber.
\end{align}
Hence, for all $t>0$,
\begin{equation}\label{FR2}
(f_{\t} g_{\t})(\mathcal{A})=f_{\t}(\mathcal{A})g_{\bm\t}(\mathcal{A}).
\end{equation}
Letting  $t\to 0$ in (\ref{FR2}) and using the strong continuity of the shifts on $\B^n$, we obtain the assertion (\ref{AsF2}).
\end{proof}

Next we prove that for ``smoothed'' functions $f \in \mathcal B_0^n$ 
the reproducing formula \eqref{repf} can be interpreted
as a $\mathcal B_0^n$-valued Bochner integral.   Here $\mathbf{1} = (1,\dots,1) \in \R^n$,
and $T(\t)$ is defined in (\ref{Tdef}).

\begin{lemma}\label{l1n}
Let $f\in \mathcal{B}_0^n$ and  $\t = (t,\dots,t) \in (0,\infty)^n$.  Define $F_\t: \C_+^n \to \B_0^n$ by
\[
F_\t(\lb) = \big(\D_n \left(f \r_{\1}^2\right)\!\!\big)(\lb) \, \r_{\t+\ov\lb}^2, \qquad \lb\in\C_+^n.
\]
Then $F_\t$ is continuous and Bochner-integrable with respect to the measure $dV_n(\lb)$.   Moreover,
\begin{equation} \label{LLn}
\left(- \frac{2}{\pi} \right)^n\int_{\C_+^n} F_\t(\lb)\,dV_n(\lb)  =  f_\t  \r_{\1+\t}^2. 
\end{equation}
\end{lemma}

\begin{proof}
The function $\lb \mapsto \r_{\t+\ov{\lb}}^2$ is continuous from $\C_+^n$ to $\B_0^n$, and hence $F_\t$ is continuous. If $\l_j=\a_j+i\beta_j$, then
\[
\|\r_{\t+\ov\lb}^2\|_{\mathcal{B}_0^n} = \prod_{j=1}^n \frac{1}{(\a_j+t)^2}.
\]
Moreover, $\big(\D_n(f r_1^2)\big)(\lb)$ is the sum of $2^n$ terms of the following form, indexed by $\O\in\P_n$, where $k= |\O|$:
\[
J_{n,\O}(\lb) := (-2)^{n-k} (\D_{\O}f)(\lb) \prod_{j\in\O} r_1(\l_j)^2 \prod_{j\in\O^c} r_1(\l_j)^3,
\]
and, by a standard estimate,
\[
|(\D_\O f)(\lb)| \le  \frac{ \|f\|_\infty}{2^k \prod_{j\in\O} \a_j}.
\]
 It follows that
\begin{align*}
\lefteqn{\hskip0pt\int_{\C_+^n} |J_{n,\O}(\lb)|\, \|\r^2_{\t+\ov{\lb}}\|_{\B_0^n}\, dV_n(\lb)} \\
&\le 2^{n-2k}\|f\|_\infty \int_{\C_+^n} \prod_{j\in\O} \frac{\a_j}{\a_j(\a_j+t)^2|1+\l_j|^2} \prod_{j\in\O^c} \frac{\a_j}{(\a_j+t)^2|1+\l_j|^3} \, dS_n(\aa,\bb) \\
&\le 2^{n-2k} \pi^n \|f\|_\infty \left(  \int_0^\infty \frac{d\a}{(\a+t)^2} \right)^n = \frac{(2\pi)^n}{4^k t^n} \|f\|_\infty.
\end{align*}
Hence
\[
\int_{\C_+^n} |\D_n(f\r_1)(\lb) | \,\|\r^2_{\t+\ov\lb}\|_{\B_0^n} \,dV_n(\lb) \le  \frac {(2\pi)^n}{t^n}
  \sum_{k=0}^n { n \choose k} 4^{-k} \|f\|_\infty=  \left(\frac{5\pi}{2t}\right)^n \|f\|_\infty.
\]

Thus, $\int_{\C_+^n}\|F_\t(\lb)\|_{\mathcal B_0^n}\,dV_n(\lb) <\infty$, so the Bochner integral of $F_t$ exists.
Since point evaluations are continuous linear functionals on $\mathcal B_0^n$, the function on the left-hand side of \eqref{LLn} maps  $\z\in\C_+^n$ to
\[
\left(- \frac{2}{\pi} \right)^n \int_{\C_+^n} \left(F_\t(\lb)\right)(\z) \,dV_n(\lb) 
= 
\int_{\C_+^n}  \K_n(\z+\t,\ov \lb) \D_n(f \r_\1^2) \,dV_n(\lb) =  (f \r_\1^2)(\z+\t),
\]
using Lemma \ref{tflip}.   Hence \eqref{LLn} holds. 
\end{proof}

Next, we apply the mapping $\Phi_\A$ to both sides of \eqref{LLn},
and then deduce the three properties of $\Phi_\A$ simultaneously.  
We say that
\[
\Upsilon:\,\mathcal{B}_0^n\mapsto L(X,X^{**})
\]
is a (bounded) ``semi-homomorphism'' if $\Upsilon$ is a bounded linear map satisfying
\begin{equation}\label{Phin}
\Upsilon(f\cdot \r_\z)=\Upsilon(f)(\z+\A)^{-1}, \qquad f \in\B_0^n, \;\z \in \C_+^n.
\end{equation}
Lemma \ref{FA2n} shows that $\Phi_A$ is a semi-homomorphism.  

\begin{lemma}\label{l2n}
Let $\A = (A_1, \dots, A_n)$  be an $n$-tuple of commuting  operators on $X$ satisfying \eqref{spr}.  Let
\[
\Upsilon:\,\mathcal{B}_0^n\mapsto \mathcal{L}(X,X^{**})
\]
be a semi-homomorphism such that  $\Upsilon(\r_\z)= \r_\z(\A), \z\in \C_{+}^n$.
Then, for all $f\in \mathcal{B}_0^n$,
\begin{equation}\label{Limn}
\Upsilon(f ) \r_\1(\A)^{2}=\lim_{t\to 0}\,Q_{t}(f;\A),
\end{equation}
in operator-norm, 
where $Q_t(f;\A)\in {L}(X)$ is defined for $t >0$ by
\begin{equation}\label{QQn}
Q_t(f;\A):= \left(- \frac{2}{\pi}\right)^n 
\int_{\C_+^n} (\D_n(f\r_\1^2))(\lb) (\A+\t+\ov\lb)^{-2}\,dV_n(\lb).
\end{equation}
Consequently, $\Upsilon$ is uniquely determined by its values on $\{\r_\z: \z \in \C_+^n\}$ and
\begin{equation}\label{LL1n}
\Upsilon(f)\in {L}(X),\quad f\in \mathcal{B}_0^n.
\end{equation}
\end{lemma}

\begin{proof}
Using (\ref{Phin}),
Lemma \ref{l1n}, and Lemma \ref{shift3},
and applying $\Upsilon$ to both sides of \eqref{LLn},
 we have
\begin{align}\label{Zn}
\Upsilon(f_\t) \r_{\1+\t}(\A)^{2}=
\Upsilon(f_\t\cdot \r_{\1+\t}^{2})
=Q_{t}(f;\A), \qquad t > 0,
\end{align}
where $Q_t(f; A)$ is given by \eqref{QQn}, and the integral converges in the norm of $L(X)$.
By the continuity of shifts on $\mathcal B_0^n$, as in Lemma \ref{C0},
\[
\lim_{t\to 0}\,\|\Upsilon(f_\t)-\Upsilon(f)\|_{{L}(X,X^{**})}=0,
\]
and, by the resolvent identity,
\[
\lim_{t\to 0}\,\|\r_{\1+\t}(\A)- \r_{\1}(\A)\|_{{L}(X)}=0.
\]
Thus, \eqref{Zn} implies (\ref{Limn}) and \eqref{LL1n}. 

Since ${\rm dom}\, (A_j)$ is dense in $X$, $r_1(A_j)^2$ has dense range for $j\in I_n$, and then $\r_\1(\A)^2$ has dense range.   The uniqueness of $\Upsilon$ follows from this and  \eqref{Limn} and (\ref{QQn}).
Since $X$ is closed in $X^{**}$,  the assertion (\ref{LL1n}) holds.
\end{proof}

Now we can conclude that $\Phi_A$, as defined in (\ref{def2}) and (\ref{deffA}),  is a $\B^n$-calculus for $\A$, as defined in Definition \ref{deffc}.

\begin{thm} \label{BAC2}
Let $\mathcal{A}=(A_1,\dots,A_n)$
satisfy the full {\rm(GSF)} condition on a Banach space $X$.   
Then $\Phi_A$ is a $\B^n$-calculus for $\A$, and it is unique.
Consequently,  $\mathcal A$ admits the $\mathcal B^n$-calculus 
if and only if $\mathcal A$ satisfies the full {\rm(GSF)} condition.
\end{thm}

\begin{proof}
Lemma \ref{l2n} shows that $\Phi_\A$ maps $\B_0^n$ into $L(X)$, and the same holds for the maps on $\B_{\O,0}^n$.  By Lemma \ref{FA2n}, $\Phi_\A$ is multiplicative on elementary functions and hence on  $\B^n$.    Then $\Phi_\A$ maps $\B^n$ into $L(X)$, and is multiplicative on all functions in $\B^n$.   In addition, $\Phi_A$ maps the function $\z \mapsto (z_j+\l)^{-1}$ to the operator $(A_j+\l)^{-1}$ for $\l\in\C_+, \, j\in I_n$, so $\Phi_\A$ is a $\B^n$-calculus for $\A$.

Let $f \in\B^n$.  Then $f \r_\1 \in \B^n_0$ and
\[
\Phi_\A(f) \r_\1(\A) = \Phi_A(f\r_\1).
\]
Since $\Phi_A(f\r_\1)$ is uniquely determined (Lemma \ref{l2n}) and the range of $\r_\1(\A)$ is dense in $X$, it follows that $\Phi_A(f)$ is uniquely determined. 

The last statement follows from the earlier statement and Corollary \ref{Bne}.
\end{proof}

\section{Possible simplifications}  \label{sect6}

In this section, we consider three possible types of simplification that might come into play when attempting to use the $\B$-calculus.   The first one would be a dramatic simplification but it is not valid in general, while the other two are valid but restricted in scope.

\subsection*{Reduction to one variable}
Let $\mathcal{A} = (A_1,\dots,A_n)$ be a commuting $n$-tuple where each $A_j \, (j \in I_n)$ has a $\B^1$-calculus.  One might hope that this automatically implies that $\A$ has a $\B^n$-calculus, but this is not valid in general. 

Let $p \in [1, \infty), p\ne2$, and let
 $\mathcal C_p$ be the $p$th Schatten-von Neumann ideal on a separable, infinite-dimensional, Hilbert space. 
Using Proposition 3 and Example 1 in \cite{Kos}, together with results from \cite{DG03} and \cite{CG08}, and the equivalence in Theorem \ref{BAC2} for $n=1$, 
it follows that there are commuting bounded operators $A_1$ and $A_2$ on  $\mathcal C_p$, each of which admit the $\B^1$-calculus, but $A_1+A_2$ does not admit the $\mathcal B^1$-calculus.    Hence, $(A_1, A_2)$  does not admit the $\mathcal B^2$-calculus, since otherwise $A_1 + A_2$  would admit the $\mathcal B^1$-calculus (see Proposition \ref{sums}).  Thus, the existence of $\mathcal B^1$-calculi for operators $A_1$ and $A_2$ separately does not imply the existence of the $\mathcal B^2$-calculus for $(A_1, A_2)$.

\subsection*{Mergers}
In some cases, two or more of the operators in an $n$-tuple $\A$ may coincide.  
We may merge some or all of the repeated operators to form an $m$-tuple  $\wt\A = (\wt A_k)_{k\in I_m}$, where there is a surjective function $\pi$ from $I_n$ to $I_m$ such that $A_j = \wt A_{\pi(j)}$ for all $j\in I_n$.   
If $\A$ satisfies (GSF), then it is trivial that $\wt\A$ also satisfies the full (GSF), because (GSF$_{\wt\O}$)  for $\wt\A$ coincides precisely with (GSF$_{\O}$) for $\A$ for some $\O\in\P_n$ where $\pi$ is a bijection from $\wt\O$ to $\O$.   
We show in this section that the $\B$-calculi for $\wt\A$ and $\A$ are compatible, thus establishing that mergers of this type within the $\B$-calculus are valid.   A similar result in a different context has been obtained  in \cite[Proposition 4.1]{AS}.

Let $f \in \B^n$, and define
\begin{equation} \label{merge}
(\Upsilon f)(\w) := f(\z),  \quad \w = (w_1,\dots,w_m) \in \C_+^m,  \; z_j = w_{\pi(j)}.
\end{equation}

\begin{prop} \label{merge2}
Let $\A$, $\wt\A$ and $\Upsilon$ be as above, and assume that $\A$ satisfies the full {\rm (GSF)} condition.   Let $f \in \B^n$.   Then $\Upsilon(f) \in \B^m$, and $(\Upsilon f)(\wt\A) = f(\A)$ in the setting of their respective $\B$-calculi.
\end{prop}

\begin{proof}
To see that $\Upsilon f \in \B^m$,  we first consider the case when $m=n-1$, and $\pi(j) = \min(j,n-1)$.  Let $g = \Upsilon f$, and
$\O \in \P_{n-1}$. 

Firstly, assume that $z_{n-1}\notin\Omega$.   
Then $(\D_\O g)(\z) = (\D_\O f)(\z,z_{n-1})$, for $\z\in\C_+^{n-1}$, and
$(H_\O g)(\aa_\O) \le (H_\O f)(\aa_\O)$,  so $H_\O g$ is integrable.   Secondly,  assume that $n-1 \in \O$.   Let $\tilde\O = (\O \setminus\{n-1\}) \cup\{n\}$.    Then
\[
(\D_{\O} g) (\z)  =  (\D_{\O}f)(\z,z_{n-1}) + (\D_{\tilde\O}f) (\z,z_{n-1}).
\]
Hence
\[
(H_{\O} g)(\aa_{\O}) \le (H_{\O}f)(\aa_{\O}) + (H_{\tilde\O}f)(\aa_{\tilde\O}).
\]
It follows that $H_{\O}g$ is integrable.   Thus $g \in \B^{n-1}$, and the map $\Upsilon : \B^n \to \B^{n-1}$ is bounded.   
By symmetry of the variables, this holds for any simple merger of just one pair of variables.   
The general case follows by carrying out $n-m$ simple mergers. 

It is clear that $\Upsilon$ is a bounded algebra homomorphism from $\B^n$ to $\B^m$.  Let $\Upsilon(f) = \Phi_{\wt\A}(\Upsilon f)$, so $\Upsilon : \B^n \to L(X)$ is a bounded algebra homomorphism.  Let $r_{\l,j}(\z) = (z_j+\l)^{-1}$ where $\z\in\C_+^n$.  Then, for $j \in I_n$ and $\l\in\C_+$, 
\[
\Upsilon r_{\l,j} = r_{\l,\pi(j)},
\]
 and
\[
\Psi(r_{\l,j}) = (\wt A_{\pi(j)}+\l)^{-1} = (A_j+\l)^{-1}.
\]
It follows from the uniqueness statement in Theorem \ref{BAC2} that $\Upsilon = \Phi_\A$, as required.
\end{proof}

\subsection*{Sums of operators}

It was observed in \cite[Lemma 2.6(5)]{BGT} that if $f\in\B^1$, $b>0$ and $g(z) = f(bz)$, then $g \in \B^1$ and $\|g\|_{\B^1} = \|f\|_{\B^1}$. 
 Moreover, if $A$ satisfies the (GSF) condition, then $g(A) = f(bA)$.   
This can be extended to the  case of several variables, and the proofs are very simple changes of variables.   
Such scalings can be combined with mergers as above.  
For example, suppose that $f \in \B^2$ and $g(z) = (b_1z, b_2 z)$, where $b_1,b_2>0$, and $A$ satisfies the (GSF) condition, then $g(A) = f(b_1A,b_2A)$.

Here we discuss the case when the operators in $\A$ may be added together, in the way that the variables $z_j$ were added in  Proposition \ref{sum}.   We present only the case where all the operators are added, but more general cases can be considered by the same techniques.

Let $\A$ be an $n$-tuple of operators which satisfy the full (GSF) condition, $f \in \B^1$ and $f^{[n]}(\z) = f(z_1 + \dots + z_n)$ for $\z \in \C_+^n$.  By Proposition \ref{sum}, $f^{[n]} \in \B^n$.   Thus $f^{[n]}(\A)$ is defined in the $\B^n$-calculus, but $f(A_1+\dots +A_n)$ is not defined in the $\B^1$-calculus (except in rare cases), because the operator $G_0 := A_1+\dots + A_n$, with domain $\dom(G_0) = \dom(A_1) \cap \dots \cap \dom(A_n)$, is not closed.   However $\dom(G_0)$ is dense in $X$ and the individual operators $-A_j$ generate commuting bounded $C_0$-semigroups $(e^{-tA_j})_{t\ge0}$, and the operators $T_\A(t) := \prod_{j=1}^n e^{-tA_j}$ form a bounded $C_0$-semigroup whose negative generator $G$ is the closure of $G_0$ (see \cite[Theorem 1.9]{Dav}).    

The following result was proved for $n=2$ in \cite[Lemma 6.2.3]{Kob}. 

\begin{prop}  \label{sums}
Let $\A$ be an $n$-tuple of operators which satisfy the full {\rm(GSF)} condition, and let $G$ be the closure of $A_1 + \dots + A_n$.   Then $G$ satisfies {\rm(GSF$_1$)} and $f(G) = f^{[n]}(\A)$ for all $f\in\B^1$.
\end{prop}

\begin{proof}
The mapping $f \mapsto f^{[n]}(\A)$ is a bounded algebra homomorphism from $\B^1$ to $L(X)$.
It suffices to prove that $r_\l^{[n]}(\A) = r_\l(G)$ for $\l\in\C_+$.  Then $G$ has a $\B^1$-calculus, so $G$ satisfies (GSF$_1$), and uniqueness of the $\B^1$-calculus implies that $f(G) = f^{[n]}(\A)$. 

 Let $\mu$ be the measure on $\C_+^n$ given by $\mu(E) = \int_{\tilde E} e^{-\l t} \,dt$, where 
$\tilde E =\{t\in\C_+: (t,t,\dots,t) \in E\}$.   Then $r_\l^{[n]}$  is the Laplace transform of the measure $\mu$.  Since the $\B^n$-calculus agrees with the HP-calculus on $\mathcal{LM}^n$,
\begin{align*}
r_\l^{[n]}(\A)x &= \int_{\R_+^n} \bigg(\prod_{j=1}^n e^{-t_jA_j}\bigg)x \, d\mu(\t) 
%&= \int_{\R_+} e^{-\l t} \prod e^{-t A_1}  \dots  e^{-t A_n} x\,dt, \qquad x\in X.
= \int_{\R_+} e^{-\l t} e^{-t G}x\,dt = r_\l(G) x
\end{align*}
for all $x \in X$.
\end{proof}

\section{Spectral mapping properties}  \label{sect7}
Let $\A$ be an $n$-tuple of commuting operators on a Banach space $X$, satisfying the full (GSF) condition, so $\A$ has a $\B^n$-calculus.   
We will use the Shilov joint spectrum of $\A$ as in \cite[Definition 9]{Mir} based on \cite[Theorem 16.3.1]{HP}.

Let $\Af$ be the (commutative) Banach subalgebra of $L(X)$ generated by all operators of the form $f(\A)$ for $f\in\B^n$ and all their resolvents.  
The spectrum of $f(\A)$ in $\Af$ coincides with the spectrum in $L(X)$, and it will be denoted by $\sigma(f(\A))$.   
For each $j \in I_n$, the Gelfand spectrum of $\Af$ splits into two subsets $\Mf_j$ and $\Nf_j$ and
there is a continuous function $\eta_j$ on $\Mf_j$ such that, for all $z\in\C_+$, 
\[
\chi((A_j+z)^{-1}) = \begin{cases}(\eta_j(\chi)+z)^{-1}   &\text{if $\chi\in\Mf_j$}, \\ 
0 &\text{if $\chi\in\Nf_j$}. \end{cases}
\]
Then $\sigma(A_j) = \eta_j(\Mf_j)$.    The \emph{Shilov joint spectrum} $\sigma(\A)$ of $\A$ is defined to be
\[
\sigma(\A) = \left\{ (\eta_1(\chi),\dots,\eta_n(\chi)) : \chi \in \cap_{j\in I_n} \Mf_j \right\}.
\]

We have the following spectral inclusion theorem.

\begin{thm}  \label{6.4.1}
Let $\A$ be an $n$-tuple of commuting operators which satisfy the full {\rm(GSF)} condition, 
and let $f \in \B^n$ and $\lb\in\sigma(\A)$.   Then $f(\lb) \in \sigma(f(\A))$.
\end{thm}

\begin{proof}
First we will assume that the resolvents of $A_j$ are bounded on the left half-plane.
It follows from this assumption that there exists $C$ such that, for all $z\in\C_+$,
\begin{equation} \label{res2}
\|(A_j+z)^{-2}(A_j+1)^{-2}\| \le C (1 + |z-1|)^{-2}, \quad j\in I_n.
\end{equation}

Let $f \in \B_0^n$.   Then \eqref{def2} and \eqref{res2} can be used to obtain
\begin{equation} \label{fr2}
f(\A) r_\1(\A)^2 =  \int_{\C_{+}^n}  \mathcal{K}_n(\A,\ov{\nub}) r_\1(\A)^2 \,(\D_n f)(\nub)\, dV_n(\nub),
\end{equation}
where the integral converges in operator norm.   

Let $\lb\in\sigma(\A)$.  There is a character $\chi$ of $\Af$ such that $\chi \in \Mf_j$ and $\chi((A_j+1)^{-1}) = (\l_j+1)^{-1}$ for all $j \in I_n$.   
Then $\chi(r_{\1}(\A)) = r_\1(\lb)$ and $\chi(\mathcal{K}_n(\A,\nub)) = \mathcal{K}_n(\lb,\nub)$.  
Applying $\chi$ to \eqref{fr2}, we obtain
\[
\chi(f(\A)) r_\1(\lb)^2 =  \int_{\C_{+}^n}  \mathcal{K}_n(\lb,\ov{\nub}) r_\1(\lb)^2 \,(\D_n f)(\nub)\, dV_n(\nub)
= f(\lb) r_\1(\lb)^2.
\]
Thus $f(\lb) = \chi(f(\A))$.

For a general $f \in \B^n$, one may consider the terms $f_{\O,0}$ in the elementary decomposition of $f$, 
using the same character $\chi$, adapting the formulas above accordingly to show that 
$f_{\O,0}(\lb_\O) = \chi(f_{\O,0}(\A_\O))$, and then summing over $\O\in\P_n$ to deduce that $f(\lb) = \chi(f(\A))$.

Now we cease to assume that the resolvents of $A_j$ are bounded on the left half-plane.  
 Instead, we can consider the operators $A_j+\ep$, where $\ep>0$, and apply the result above.  
The same $\chi$ as above may be used for sufficently small $\ep>0$ to show that 
$f(\lb+\bm\ep) = \chi(f(\A+\bm\ep))$, where $\bm\ep = (\ep,\dots,\ep)$.  
 Letting $\ep\to0$ and using Lemmas \ref{C0} and \ref{shift3}, it follows that $f(\lb) = \chi(f(\A)) \in \sigma(f(\A))$.
\end{proof}

There are other results on spectral mappings for the $\B^n$-calculus, as follows.

\begin{prop} \label{sap}
Let $\A$ be an $n$-tuple of commuting operators on a Banach space $X$, satisfying the full {\rm(GSF)} condition, and let $f \in \B^n$.   The following results hold.
\begin{enumerate} [\rm1.]
\item  If $x\in X$ is an eigenvector of $A_j$ with eigenvalue $\l_j$ for each $j\in I_n$, then $x$ is an eigenvector for $f(\A)$ with eigenvalue $f(\lb)$.
\item If $(x_k)_{k\in\N}$ are unit vectors in $X$ forming an approximate eigenvector of $A_j$ with approximate eigenvalue $\l_j$ for each $j\in I_n$, then $f(\lb)$ is an approximate eigenvalue of $f(\A)$.
\end{enumerate}
\end{prop}

\begin{proof}
The first statement can be seen very easily from the definition of  $f(A)$ in (\ref{OperA1}).

The second statement follows from the first, using an $F$-product as in \cite{Nag}.   For this, let
\begin{align*} 
\ell^\infty_T(X) &= \Big\{(x_k)_{k\in\N}:  \sup_{k\in\N} \|x_k\|< \infty,  \; \lim_{t\to0} \sum_{j=1}^n \sup_{k\in\N} \left\|e^{-tA_j}x_k - x_k\right\|=0\Big\}, \\
c_0(X) &= \big\{(x_k)_{k\in\N}:  \lim_{k\to\infty} \|x_k\| =0\big\}.
\end{align*}
Then $c_0(X)$ is a closed subspace of $\ell^\infty_T(X)$.   Let $\mathcal{X}_T  = \ell^\infty_T(X)/c_0(X)$, and $\pi: \ell^\infty_T(X) \to \mathcal{X}_T$ be the quotient map.   
For each $j\in I_n$, the $C_0$-semigroup $(e^{-tA_j})_{t\ge0}$ induces a $C_0$-semigroup on $\mathcal{X}_T$.  Let $-Z_j$ be the generator, and $\mathcal{Z} = (Z_1,\dots,Z_n$). 
The operators $f(\mathcal{Z})$ on $X$ induce operators on $\mathcal{X}_T$ which form a $\B^n$-calculus for $\mathcal{Z}$.  

There are standard methods to show that (a) an approximate eigenvector $(x_k)_{k\in\N}$ for $\A$ is mapped by $\pi$ to an approximate eigenvector for $\mathcal{Z}$ with eigenvalues $\lb$, and (b) $\lb$ is an eigenvalue for $\mathcal{Z}$. By the first statement (1), $f(\lb)$ is an eigenvalue of $f(\mathcal{Z})$.  It then follows directly that $\lb$ is an approximate eigenvalue of $f(\A)$.    Details for (a) and (b) in the cases $n=1$ and $n=2$ may be found in \cite[pp.\ 20, 78]{Nag} and \cite[Propositions 2.3.17, 6.4.4]{Kob}.
\end{proof}

If all the operators $A_j$ are sectorial of angle less than $\pi/2$, then $\A$ satisfies the full (GSF) condition and the formula \eqref{def2} can be adapted to
\begin{equation} \label{fsec}
f(\A)  =  \int_{\C_{+}^n}  \mathcal{K}_n(\A,\ov{\nub}) \,(\D_n f)(\nub)\, dV_n(\nub),
\end{equation}
where the integral is convergent in operator norm.  

\begin{thm} \label{6.4.3}
Let $\A$ be an $n$-tuple of commuting sectorial operators of angle less than $\pi/2$ on a Banach space $X$, and let $f \in \B^n$.  Then
\[
\sigma(f(\A)) \subseteq \bigcup_{\O\in\P_n} f_{\O} (\sigma(\A)).
\]
\end{thm}

\begin{proof}
Let $\chi$ be a character of $\Af$, and $\O_\chi = \{j\in I_n:  \chi \in \Mf_j\}$.   
For $j \in \O_\chi$, let $\l_j$ satisfy $\chi((z+A_j)^{-1}) = (z+\l_j)^{-1}$ for all $z\in\C_+$.
 
Let $\O \in \P_n$, and $k= |\O|$.    Assume first that $\O \subseteq \O_\chi$.  Then $\lb_{\O} \in \sigma(\A_{\Omega})$, and $\chi(\K_k(\A_\O,\ov{\nub})) = \K_k(\lb_\O,\ov{\nub})$.  It follows from \eqref{fsec} that
\[
\chi(f_{\O,0}(\A))  = \int_{\C_+^k} \mathcal{K}_k(\lb_\O,\ov{\nub}) \,(\D_k f_{\O,0})(\nub)\, dV_n(\nub) = f_{\O,0}(\lb_\O).
\]
Note that this conclusion is valid even if $\Omega$ is the empty set.   Then $f_{\O,0}$ is a scalar.

If $\O$ is not contained in $\O_\chi$, then, 
for some $j \in \O$, $\chi(r_z(A_j))=0$ for all $z\in\C_+$, so $\chi(\mathcal{K}_k(\A,\ov{\nub}))=0$ and $\chi(f_{\O,0}(\A))=0 = f_{\O,0}(\lb)$.

Summing over all $\O\in\P_n$, we obtain that $\chi(f(\A)) = f_{\O_\chi}(\lb)$.
\end{proof}

\begin{cor}
Let $\A$ be an $n$-tuple of commuting sectorial operators of angle less than $\pi/2$ on a Banach space $X$, and let $f \in \B_0^n$.  Then
\[
\sigma(f(\A)) = f(\sigma(\A)).
\]
\end{cor}

\begin{proof}
Theorem \ref{6.4.1} shows that $f(\sigma(\A)) \subseteq \sigma(f(\A))$.   For the reverse inclusion, $f_{\O,0} = 0$ unless $\O=\P_n$.  Then the proof of Theorem \ref{6.4.3} shows that $\chi(f(\A))=f(\lb)$ for some $\lb\in\sigma(\A)$, as required.
\end{proof}

\section{Estimates for functions and operators}  \label{sect8}
Let $\mathcal A=(A_1, \dots, A_n)$ be an $n$-tuple on commuting operators satisfying
the full (GSF) condition. In this section, we will use the $\mathcal B^n$-calculus constructed in Theorem \ref{BAC2}
to provide several  estimates for $f(\mathcal A)$ in terms of $\|f\|_{\infty}$
for $f$ belonging to substantial subclasses of $H^\infty(\mathbb C^n_+)$.
The estimates resemble the situation of a bounded $H^{\infty}$-calculus and thus could be potentially useful.

We will need a Fourier characterisation of functions in $H^\infty(\mathbb C^n_+)$
in terms of their boundary values.
If $f \in H^\infty(\mathbb C^n_+)$, then $f^b(\bb) := \lim_{{\bm \alpha} \to 0+} f({\bm \alpha}+i{\bm \beta})$ exists for almost all ${\bm \beta} \in \mathbb R^n$ (see \cite[Theorem F]{Qian}). 
Then $f^b \in L^\infty(\mathbb R^n_+)$.
For $f \in L^\infty (\mathbb R^n_+),$ define the spectrum of $f$ as
${\rm sp}(f):={\rm supp} \, (\mathcal F^{-1}f),$ where $\mathcal F$ stands for the
(distributional) Fourier transform
of $f$.  
Let $H^\infty(\mathbb R^n):=\{f \in L^\infty(\mathbb R^n): {\rm sp}\, (f) \subseteq \mathbb R^n_+\}$.
Then $H^{\infty}(\mathbb C^n_+)$ is isometrically isomorphic to $H^\infty(\mathbb R^n)$ via the mapping 
$f \to f^b$,
and the inverse mapping is given by
\begin{equation}\label{poisson}
 f({\bm \alpha}+i{\bm \beta})=(P_{{\bm \a}}*f^b)({\bm \b})=\int_{\mathbb R^n} P({\bm \alpha},{\bm \beta -\t}) f^b(\mathbf{t})\, d\mathbf{t},
\end{equation}
where ${\bm \alpha} \in \mathbb R^n_+, \, {\bm \beta} \in \mathbb R^n$, and
\[P_{{\bm \a}}({\bm \b})=P({\bm \alpha},{\bm \beta}) =  \prod_{j=1}^n \frac{1}{\pi} \frac{\a_j}{\alpha_j^2+ \beta_j^2}
\] 
stands for the Poisson kernel.
This fact is well-known for $n=1$ (see \cite[Section II.1.5]{Havin}).
For $n >1$ it follows from \cite[Theorem 2.7]{Qian}. 
 
For $\tau>0$, let  $H^\infty([\tau, \infty)^n):=\{f \in H^\infty(\mathbb C_+^n): {\rm sp}\, (f^b) \subseteq 
[\tau,\infty)^n\}$,
and 
\begin{equation*}
e_\tau(\mathbf{z})=e^{-\tau(z_1+\dots+z_n)}, \qquad \mathbf{z} \in \mathbb C_+^n.
\end{equation*}
Then
\begin{equation} \label{exp}
H^{\infty}([\tau, \infty)^n)=e_{\tau} H^\infty(\mathbb C^n_+),
\end{equation}
so that from \eqref{poisson}, for $f \in H^\infty([\tau, \infty]^n)$,
\begin{equation}\label{bound}
|f(\z)| \le e^{-\tau (\sum_{i=1}^n {\rm Re}\, z_i)}\|f^b\|_{\infty},  \qquad \z\in\C_+^n. 
\end{equation}
Indeed, $(e_{-\tau} f)^b \in L^\infty(\mathbb R^n)$ and
${\rm sp}\, (e_{-\tau}f) \subseteq \mathbb R^n_+$.   If $g({\bm \alpha}+i{\bm \beta})=(P_{{\bm \a}}*(e_{-\tau}f)^b)({\bm \b}),$ ${\bm \alpha}+i{\bm \beta} \in \mathbb C^n_+$,
 then by the above $g \in H^\infty(\mathbb C_+^n)$ and  
 $(e_{\tau}g)^b = f^b$.  Since both $f$ and $e_{\tau}g$ are in $H^\infty(\mathbb C^n_+)$
we infer from \eqref{poisson} that $e_{\tau}(\mathbf{z}) g(\mathbf{z})=f(\mathbf{z})$
for all $\mathbf{z} \in \mathbb C^n_+$.  Taking into account \eqref{poisson}, we obtain \eqref{bound}
which is well-known for $n=1$.
It follows from a standard Fourier characterisation of $H^\infty(\mathbb C_+)$
and it was used essentially in \cite{BGT}, but
its multivariate counterpart seems not to have been noted in the literature.

We will use an elementary variant of Bernstein's inequality, as follows (see \cite[Theorem 7.3.1]{Horm} or \cite[Chapter 3, p.116]{Nik}).  If $f \in H^\infty([0, \sigma]^n)$
then $f$ is an entire function of exponential type $(\sigma, \dots, \sigma)$, as defined in \cite[p.98]{Nik} and 
for all $\mathbf{z}={\bm \alpha}+i{\bm \beta} \in \mathbb C^n_+$ and $\Omega \in \mathcal P_n$,
\begin{equation}\label{Des1}
\sup_{{\bm \beta} \in \mathbb R^n}|\mathcal D_{\Omega} f({\bm \alpha}+i{\bm \beta})|\le
\sigma^{|\Omega|} \sup_{{\bm \beta} \in \mathbb R^n}\,|f({\bm \alpha}+i{\bm \beta})|.
\end{equation}

In the following lemmas, we show that functions of various classes belong to $\B_0^n$, and we give estimates for their norms.  Then Proposition \ref{Pr1} shows that $f\in\B^n$ and $\|f\|_{\B^n} \le 2^n\|f\|_{\B_0^n}$.   An alternative estimate of the $\B^n$-norms can be obtained by an extension of the method used for the $\B_0^n$-norms.   

\subsection*{Holomorphic extensions to the left}
We will show here that if  
$f\in \mathcal B^n$ extends to a bounded analytic function in a larger half-space, then some damped versions of $f$ 
are in $\mathcal B_0^n$ with norm dominated by an $H^\infty$-norm of $f$.   For $n=1$, functions of this type were considered in \cite[Section 5.2]{BGT}.

Let $\omega>0$ and 
$H^\infty_\omega=H^{\infty}(\{\mathbf{z}\in \mathbb C^n: \Re z_j >-\omega, \, j\in I_n\})$.   If $f\in H^\infty_\omega$ and  $\O\in\P_n$, then Cauchy's inequality gives
\begin{equation}\label{Cau}
|\mathcal{D}_\Omega f(\mathbf{z})|\le \frac{\|g\|_{H^\infty_\omega}}{2^k
|\omega+{\bm \alpha}_\Omega|},  \qquad \mathbf{z}={\bm \alpha}+i{\bm \beta} \in \mathbb C^n_+,
\end{equation}
where
$k=|\O|$ and
$|\omega+{\bm \alpha}_\Omega|=\prod_{j\in\O}
(\omega+\a_{j})$.

\begin{lemma}\label{AN}
Let $f\in H^\infty_\omega$, $\omega>0$ and let
$\nu>0$, $\lambda\in \C_{+}$, $m= \min\{\omega,\Re\l\}$.   
Let
\begin{equation} \label{RSl}
R^\nu_{\lambda}(\mathbf{z}):=\prod_{j=1}^n (\lambda+z_j)^{-\nu}, \qquad 
S^\nu_{\lambda}(\mathbf{z}):=(\lambda+z_1+\cdots+z_n)^{-\nu}, \quad \mathbf{z} \in \mathbb C_+^n.
\end{equation}
Then $R^\nu_{\lambda}f \in \B^n_0$ and $S^\nu_{\lambda}f \in \mathcal B^n_0$.   Moreover, 
\begin{align*}\label{LR}
\|R^\nu_{\lambda}f\|_{\mathcal{B}_0^n}\le
\|f\|_{H^\infty_\omega}
\left(\frac{1}{2\nu}+\frac{1}{m^\nu}\right)^n, \quad
%&\|R^\nu_{\lambda}f\|_{\mathcal{B}^n}\le
%\|f\|_{H^\infty_\omega}
%\left(\frac{1}{2\nu}+\frac{1}{m}+1\right)^n,  \\
\|S^\nu_{\lambda}f\|_{\mathcal{B}_0^n}\le
 \frac{\|f\|_{H^\infty_\omega}}{m^\nu}
\left(\frac{n}{2\nu}+1\right)^n.% \quad 
%&\|S^\nu_{\lambda}f\|_{\mathcal{B}^n}\le
% \frac{\|f\|_{H^\infty_\omega}}{m^\nu}
%\left(\frac{n}{2\nu}+2\right)^n.
\end{align*}
\end{lemma}

\begin{proof}
From  \eqref{Bnnorm} and \eqref{Cau}, we have
\begin{align*}
\|R_{\lambda}^\nu f\|_{\mathcal{B}_0^n}&\le
\sum_{\Omega\in \mathcal{P}_n}
\int_{\R_{+}^n}\,\sup_{\z\in W_{\aa_{I_n}}}\,|(D_\O R_{\lambda}^\nu)({\z})||(D_{\O^c}f)(\z)|\,d{\bm \alpha}\\
&\le\|f\|_{H^\infty_\omega}\sum_{\O\in\P_n}
\frac{\nu^{|\O|}}{2^{n-|\O|}}
\int_{\R^{|\Omega|}_{+}}\frac{d{\bm \alpha}}{|m+\aa_\O|^{1+\nu}}\\
&=\|f\|_{H^\infty_\omega}  \sum_{j=0}^n
\frac{\binom{n}{j}}{(2\nu)^{n-j}} \frac{1}{m^{\nu j}}
%&=\|f\|_{H^\infty_\omega}\sum_{k=0}^n tinom{n}{k}\sum_{j=0}^k tinom {k}{j}
%\frac{1}{(2\nu)^{k-j}m^{\nu k}}
%= \|f\|_{H^\infty_\omega}\sum_{k=0}^n tinom{n}{k}
%\left(\frac{1}{2\nu}+\frac{1}{m}\right)^k\\
=  \|f\|_{H^\infty_\omega}
\left(\frac{1}{2\nu}+\frac{1}{m^\nu}\right)^n.
\end{align*}

By a similar argument, we have (with $d\t_{n-j} = dt_1\dots dt_{n-j}$),
\begin{align*}
\|S_{\lambda}^\nu f\|_{\mathcal{B}_0^n}
&\le
\|f\|_{H^\infty_\omega}
\sum_{j=0}^n \frac{\binom{n}{j}}{2^{n-j}}\prod_{k=0}^{j-1}(\nu+k)
\int_{\R_{+}^n}
\big(m+\sum_{k=1}^n t_k\big)^{-(\nu+j)}\prod_{k=1}^{n-j}(m+t_k)^{-1} \,d\t\\
&=\|f\|_{H^\infty_\omega}
\sum_{j=0}^n \frac{\binom{n}{j}}{2^{n-j}}
\int_{\R_{+}^{n-j}}\big(m+\sum_{k=1}^{n-j} t_k\big)^{-\nu}\prod_{k=1}^{n-j}(m+t_k)^{-1} \,d\t_{n-j}.
\end{align*}
By replacing $t_k$ by $m t_k$, taking the inequality $\prod_{k=1}^{n-j} (1+t_k) \le \big(1 + \sum_{k=1}^{n-j} t_k\big)^{n-j}$ and raising it to the power $-\nu/(n-j)$, we arrive at
\begin{align*}
\|S_{\lambda}^\nu f\|_{\mathcal{B}_0^n}
&\le \frac{\|f\|_{H^\infty_\omega}}{m^\nu}
\sum_{j=0}^n \frac{\binom{n}{j}}{2^{n-j}}
\int_{\R_{+}^{n-j}}
\prod_{k=1}^{n-j}(1+t_k)^{-(1+\nu/(n-j))} \,d\t_{n-j}\\
&=\frac{\|f\|_{H^\infty_\omega}}{m^\nu}
\sum_{j=0}^n \binom{n}{j}\left(\frac{n-j}{2\nu}\right)^{n-j} 
\le \frac{\|f\|_{H^\infty_\omega}}{m^\nu}
\left(\frac{n}{2\nu}+1\right)^n.  \qedhere
\end{align*}
\end{proof}

%\subsection{Exponent}

Now we show that the damping in Lemma \ref{AN} is not necessary if ${\rm sp}(f)$ is separated
from zero.

\begin{lemma}\label{BnE}
Let $f\in H^\infty([\tau,\infty)^n)\cap H^\infty_\omega$, $\tau>0$, $\omega>0$.
Then $f \in \mathcal B^n_0$, and
\begin{equation}\label{LE}
\|f\|_{\mathcal{B}_0^n}
\le\|f\|_{H^\infty_\omega}e^{-n\omega \tau}\left(1+
\frac{1}{2}\log\left(1+\frac{1}{\tau\omega}\right)\right)^n.
%\frac{1}{2}\log\left(1+\frac{1}{\tau\omega}\right)\right)^n.
\end{equation}
\end{lemma}

\begin{proof} 
By \eqref{bound},
there exists $g \in H^\infty_\omega$ such that
$
f(\mathbf{z})=e^{-\tau(z_1+\cdots+z_n)}g(\mathbf{z}), \, \z\in\C_+^n,
$
and thus
$\|g\|_{H_\omega^\infty}=e^{-n\omega\tau}\|f\|_{H^\infty_\omega}$.
From this and \eqref{Cau},  we infer that 
\begin{align*}
|\mathcal{D}_n f(\mathbf{z})|&\le
\sum_{\O\in\P_n} |(D_{\O} e^{-\tau(z_1+\cdots+z_n)})(D_{\O^c} g(\mathbf{z}))|\\
&\le e^{-\tau(\alpha_1+\cdots+\alpha_n)}\sum_{\O\in\P_n}\left(\frac{\tau^{|\O|}}{2^{n-|\O}|\omega+{\aa}_{\O^c}|}\right)
\|g\|_{H^\infty_\omega},
\end{align*}
where $z_j=\a_j+i\b_j$.  Since the variables in the integrals below over $\R_+^{n}$ split as the product of $j$ integrals of one function and $n-j$ integrals of another function, we see that
\begin{align*}
\|f\|_{\mathcal{B}_0^n}
&\le \|g\|_{H^\infty_\omega}\sum_{\Omega\in \mathcal{P}_n}
\,\frac{\tau^{|\O|}}{2^{n-|\O|}}
\int_{\R_{+}^n}e^{-\tau \sum_{j=1}^n \alpha_j}\left(\frac{1}{|\omega+{\bm \alpha}_{\O^c}|}\right)\,d{\bm \alpha}\\
&= \|g\|_{H^\infty_\omega}\sum_{j=0}^n \binom{n}{j}
\frac{1}{2^{n-j}} \left( \int_0^\infty \tau e^{-\tau\a} \,d\a \right)^j \left( \int_0^\infty e^{-\tau\a} (\o+\a)^{-1} \,d\a \right)^{n-j}.
\end{align*}
The first integral in the line above is equal to $1$, and the second integral is shown in the proof of \cite[Lemma 3.2(2)]{BGT} to be less than $\log (1 + (\tau\o)^{-1})$, so 
%&\null\hskip-20pt \le \|g\|_{H^\infty_\omega}\sum_{j=0}^n\binom{n}{j} \frac{1}{2^{n-j}}
%\left(\log\left(1+\frac{1}{\tau\omega}\right)\!\right)^{n-j}
\begin{align*}
\|f\|_{\B_0^n}&\le \|g\|_{H^\infty_\omega}\sum_{j=0}^n \binom{n}{j}
\!\left(\frac{1}{2}\log\left(1+\frac{1}{\tau\omega}\right)\!\right)^{n-j}
\!= \|g\|_{H^\infty_\omega}\!\left(1+
\frac{1}{2}\log\left(\!1+\frac{1}{\tau\omega}\right)\!\right)^n \!\!. \qedhere
\end{align*}  
\end{proof}

\subsection*{Functions with restricted spectrum}
Instead of assuming that $f \in H^\infty_\omega$, 
we now assume that ${\rm sp}(f)$ is separated from zero and infinity
and we obtain an estimate of $\|f\|_{\mathcal B^n}$ in terms of $\|f\|_{\infty}$. 
We need the following technical lemma.    For $\t = (t_1,\dots,t_k) \in (0,\infty)^k$, let $S_k(\t)=  \sum_{j=1}^k t_j$ and $P_k(\t) = \prod_{j=1}^k t_j$.

\begin{lemma}\label{Fin}
For $k\in \N$ and $a\in(0,1)$, let
\[
J_k(a):=\int_{\R_{+}^k}\,
\frac{d\t}
{P_k(\t) +a e^{S_k(\t)}}.
\]
Then
\begin{equation}\label{An}
J_k(a)\le 2^{2k}\log^k\left(1+\frac{1}{a}\right).
\end{equation}
\end{lemma}

\begin{proof} 
We argue by induction. If $k=1,$ then \eqref{An} has been proved in \cite[p.37, (2.12)]{BGT}.
Next, assume that \eqref{An} holds for some $k\ge 1$.    Note that $J_k$ is a decreasing function, and by Fubini's theorem,
\[
J_{k+1}(a) = \int_0^\infty \int_{\R_+^k} \frac{d\t\,d\tau}{P_k(\t)\tau + a e^{S_k(\t)}e^\tau} = \int_0^\infty J_k(ae^\tau/\tau) \,\frac{d\tau}{\tau}.
\]
Using the monotonicity of $J_k$, a change of variable $u=e^{-\tau}$, and the inequalities $tJ_k(t)\le 1$,
 and $(4\log 2)^{k+1}\ge 2(k+1)$ if $k\ge 1$,
we infer that
\begin{align*}
J_{k+1}(a)&\le\int_0^1 J_k(a/\tau)\frac{d\tau}{\tau}+
\int_1^\infty J_k(ae^{\tau})\,d\tau
=\int_0^1 J_k(a/\tau)\frac{d\tau}{\tau}+
\int_0^{1/e} J_k(a/u)\frac{du}{u}\\
&\le 2\int_0^1 J_k(1/s)\frac{ds}{s}+
2\int_1^{1/a} J_k(1/s)\frac{ds}{s}
\le 2 +
2^{2k+1}\int_1^{1/a}\log^k(1+s)\,\frac{ds}{s}\\
&\le 2+
2^{2k+2}\int_1^{1/a}\frac{\log^k(1+s)}{1+s}\,ds
=2+\frac{2^{2k+2}}{k+1}\left(\log^{k+1}(1+1/a)-\log ^{k+1}2\right)\\
&\le 2^{2(k+1)}\log^{k+1}(1+1/a),
\end{align*}
so \eqref{An} is true for $k+1$. The assertion follows.
\end{proof}

The following lemma was proved in \cite[Lemma 2.5]{BGT} for $n=1$ and in \cite[Lemma 4.1.1]{Kob} for $n=2$.   The functions considered in those cases played crucial roles in those papers.

\begin{lemma}\label{Log}
Let $f\in H^\infty([\epsilon,\sigma]^n)$, where $0<\epsilon<\sigma$. Then
$f\in \mathcal{B}^n_0$ and
\begin{equation}\label{LL}
\|f\|_{\mathcal{B}_0^n}\le
2^{n+1}\|f\|_\infty
\left( \log\left(1+\left(\frac{2\sigma}{\epsilon}\right)^n\right)\right)^n. %\quad
%\|f\|_{\mathcal{B}^n}\le
%2\|f\|_\infty
%\left(
% 2\log\left(1+\left(\frac{2\sigma}{\epsilon}\right)^n\right)+1\right)^n.
\end{equation}
\end{lemma}

\begin{proof}
From (\ref{Des1}),
we have 
\begin{equation}\label{Des}
|\D_n f({\bm \alpha}+i{\bm \beta})|\le
\sigma^n \sup_{\bm\beta \in \R^n}\,|f({\bm \alpha}+i{\bm \beta})|
\le \sigma^n e^{-\epsilon S(\aa)} \|f\|_\infty.
\end{equation}
From Cauchy's inequality,
\[
|\D_{n}f({\bm \alpha}+i{\bm \beta})|\le \frac{\|f\|_\infty}{2^n P_n(\aa)}.
\]
Thus
\begin{align*}
|\D_n f({\bm \alpha}+i{\bm \beta})|&\le \|f\|_\infty
\min\left\{\sigma^n  e^{-\epsilon S_n(\aa)},
2^{-n} P_n(\aa)^{-1}\right\} \le \frac{2\|f\|_\infty}
{2^n P_n(\aa)
+\sigma^{-n}e^{\epsilon S_n(\aa)}}.
\end{align*}
From Lemma, \ref{Fin},
\begin{align*}
\int_{\R_{+}^n}\,
\frac{d{\bm \alpha}}{P_n(\aa) 
+\sigma^{-n}e^{\epsilon S_n(\aa)}}
=J_n\left(\left(\frac{\epsilon}{2\sigma}\right)^n\right)
\le 2^{n} \left(\log\left(1+\left(\frac{2\sigma}{\epsilon}\right)^n\right)\right)^n.
\end{align*}
Now \eqref{LL} follows.
\end{proof}

Lemmas \ref{AN}, \ref{BnE} and \ref{Log} imply the following operator norm-estimates for 
$\mathcal B^n$-functions of $n$-tuples of operators satisfying the full (GSF) condition.
They were obtained in  \cite{BGT} for $n=1$, providing direct alternative proofs 
for the results in \cite{Haase_JFA}, \cite{Haase_JFAR}, \cite{Vit} and their improvements.
The estimates quantify a deviation of the $\mathcal B^n$-calculus from the bounded $H^\infty$-calculus, 
and they simplify operator norm-estimates in several situations of interest.    In statements (i) and (ii), $f(\A)$ is not necessarily a bounded operator, but it can be defined as a closed operator in the extended half-plane calculus for several variables.   The proofs are consequences of Lemmas \ref{AN}, \ref{BnE} and \ref{Log}.   For (ii), one uses (\ref{exp}).

\begin{thm}
Let $\mathcal A$ be an $n$-tuple of operators on a Banach space $X$, satisfying
the full {\rm (GSF)} condition, and let $\gamma_{\mathcal A,0}$ be the norm of $\Phi_A$ on $\B_0^n$.
Then the following hold.

\begin{enumerate}[\rm(i)]
\item
Let $f\in H^\infty_\omega$, $\omega>0$, and let
$\nu>0$, $\lambda\in \C_{+}$, and
$m=\min\{\omega,{\rm Re}\,\lambda\}$. Let $R_\l^\nu$ and $S_\l^\nu$ be as in {\rm(\ref{RSl})}.  Then
\[
\|(R_\l^\nu f)(\A)\| \le
\gamma_{\mathcal A, 0} \left(\frac{1}{2\nu}+\frac{1}{m}\right)^n \|f\|_{H_\omega^\infty}.
\]
and
\[
\|(S_\l^\nu f)(\A)\|\le \gamma_{\mathcal A, 0}
m^{-\nu}\left(\frac{n}{2\nu}+1\right)^n \|f\|_{H^\infty_\omega}.
\]
\item
Let $g \in H^\infty_\omega$, $\omega>0$ and $\tau>0$.   Then
\[
\|(g e_{-\tau})(\A)\|
\le \gamma_{\mathcal A, 0} e^{-n\omega \tau}\left(1+
\frac{1}{2}\log\left(1+\frac{1}{\tau\omega}\right)\right)^n \|g\|_{H^\infty_\omega}.
\]
\item Let $f\in H^\infty([\epsilon,\sigma]^n)$, $0<\epsilon<\sigma$.  Then
\begin{equation}\label{LL1}
\|f(\mathcal A)\| \le
2^{n+1}\gamma_{\mathcal A, 0} \left(
 \log\left(1+\left(\frac{2\sigma}{\epsilon}\right)^n\right)\right)^n \|f\|_\infty.
\end{equation}
\end{enumerate}
\end{thm}

\section{Appendix}

For the sceptical reader,  we will justify the application of Fubini's theorem
in the proof of  Lemma \ref{FA2n} by showing that
\begin{equation*}\label{PrF2}
J_\t:=\int_{\C_{+}^m}\int_{\C_{+}^k} |(\D_{\O_f} f)(\lb)|\,
|(\D_{\O g}g)(\nunu)| \,
S_\t(\lb,\nunu) \, dV_k(\lb)dV_m(\nunu)<\infty,
\end{equation*}
where $f$ and $g$ are elementary functions in $\B^n$ with supports $\O_f$ and $\O_g$ with cardinalities $k$ and $m$ respectively, and $\O_f \cup \O_g = I_n$.   Here $\lb$ and $\nunu$ are indexed by $\O_f$ and $\O_g$, respectively, and
\[
S_\t(\lb,\nunu)=\int_{\C_{+}^n}|\langle \mathcal{K}_n(\mathcal{A},\ov{\z})x,x^{*}\rangle|  
\left|\mathcal{D}_{n}\left(\mathcal{K}_k(\z_{\O_f}+\t_{\O_f},\ov{\lb})
\mathcal{K}_m(\z_{\O_g}+\t_{\O_g},\ov{\nunu})\right)\right| dV_n(\z),
\]
where $\D_n$ denotes differentiation with respect to all the $z_j$ variables (once each).

We may assume that $\|x\|=\|x^{*}\|=1$ and put $\b_j = \Re\l_j$, $\eta_l = \Re\nu_l$.  
Using the (GSF) condition for $\A$, we can estimate:
\begin{align*}
S_\t(\lb,\nunu) &\le \gamma_\A\int_{\R_{+}^n}
\sup_{\Re z_j=\a_j} \left|\mathcal{D}_{n}(\mathcal{K}_k(\z_{\O_f}+ \t_{\O_f}, \ov\lb)
\mathcal{K}_m(\z_{\O_g}+ \t_{\O_g},\ov{\nunu}))\right| \, d\aa\\
&\le \gamma_\A \bigg|\int_{\R_{+}^n}
\D_n \left(\mathcal{K}_k(\aa_{\O_f}+\bm{t}_{\O_f},\bb_{\O_f})
\mathcal{K}_m(\aa_{\O_g}+\bm{t}_{\O_g},\bm{\eta}_{\O_g})\right)\,d\aa \bigg|\\
&=\gamma_{\mathcal{A}} \big|\mathcal{K}_k(\bm{t}_{\O_f},\bb_{\O_f}) \mathcal{K}_m(\bm{t}_{\O_g},\bm{\eta}_{\O_g})\big|
\le \gamma_{\mathcal{A}}\left(\frac{2}{\pi}\right)^{k+m}\prod_{j\in\O_f} \frac{1}{2t\beta_j}
\prod_{l\in\O_g} \frac{1}{2t\eta_l}.
\end{align*}
The inequality in the second line is obtained by straightforward estimates on the integrand on the function.  
Moving the absolute value signs outside the integral is justified, because the integrand is real-valued, and its sign depends only on $k$ and $m$.    
The equality in the third line comes from the fundamental theorem of calculus for each of the variables.
The final inequality is straightforward.
Now
\begin{align*}
J_\t&\le
\frac{\gamma_{\mathcal{A}}}{(\pi t)^{k+m}}
\int_{\R_{+}^m}\int_{\R_{+}^k}\sup_{\Re\l_j=\beta_j,\, \Re\nu_l = \eta_l}\,
|(\mathcal{D}_{\O_f} f)(\lb)| 
|(\mathcal{D}_{\O_g}g)(\nunu)|\,
d\bb d\bm{\eta}\\
&= \frac{\gamma_{\mathcal{A}}}{(\pi t)^{k+m}}
\|f\|_{\mathcal{B}_{\O_f,0}}\, \|g\|_{\mathcal{B}_{\O_g,0}}.
\end{align*}

 \end{document}